\documentclass[12pt]{article}
\usepackage{latexsym,amsfonts,amssymb,amsmath,amsthm,ulem}
\usepackage{graphicx}
\usepackage{enumitem}
\usepackage{comment}
\DeclareGraphicsExtensions{.pdf,.jpeg,.png,.eps}
\usepackage{caption}
\usepackage{epsfig,epstopdf,amssymb,color}
\usepackage[capposition=bottom]{floatrow}
\usepackage{setspace} 
\allowdisplaybreaks

\begin{document}
\setlength{\baselineskip}{16pt}

\parindent 0.5cm
\evensidemargin 0cm \oddsidemargin 0cm \topmargin 0cm \textheight 22cm \textwidth 16cm \footskip 2cm \headsep
0cm

\newtheorem{theorem}{Theorem}[section]
\newtheorem{lemma}{Lemma}[section]
\newtheorem{proposition}{Proposition}[section]
\newtheorem{definition}{Definition}[section]
\newtheorem{example}{Example}[section]
\newtheorem{corollary}{Corollary}[section]

\newtheorem{remark}{Remark}[section]
\numberwithin{equation}{section}

\def\p{\partial}
\def\I{\textit}
\def\R{\mathbb R}
\def\C{\mathbb C}
\def\u{\underline}
\def\l{\lambda}
\def\a{\alpha}
\def\O{\Omega}
\def\e{\epsilon}
\def\ls{\lambda^*}
\def\D{\displaystyle}
\def\wyx{ \frac{w(y,t)}{w(x,t)}}
\def\imp{\Rightarrow}
\def\tE{\tilde E}
\def\tX{\tilde X}
\def\tH{\tilde H}
\def\tu{\tilde u}
\def\d{\mathcal D}
\def\aa{\mathcal A}
\def\DH{\mathcal D(\tH)}
\def\bE{\bar E}
\def\bH{\bar H}
\def\M{\mathcal M}
\renewcommand{\labelenumi}{(\arabic{enumi})}

\def\disp{\displaystyle}
\def\undertex#1{$\underline{\hbox{#1}}$}
\def\card{\mathop{\hbox{card}}}
\def\sgn{\mathop{\hbox{sgn}}}
\def\exp{\mathop{\hbox{exp}}}
\def\OFP{(\Omega,{\cal F},\PP)}
\newcommand\JM{Mierczy\'nski}
\newcommand\RR{\ensuremath{\mathbb{R}}}
\newcommand\CC{\ensuremath{\mathbb{C}}}
\newcommand\QQ{\ensuremath{\mathbb{Q}}}
\newcommand\ZZ{\ensuremath{\mathbb{Z}}}
\newcommand\NN{\ensuremath{\mathbb{N}}}
\newcommand\PP{\ensuremath{\mathbb{P}}}
\newcommand\abs[1]{\ensuremath{\lvert#1\rvert}}
\newcommand\normf[1]{\ensuremath{\lVert#1\rVert_{f}}}
\newcommand\normfRb[1]{\ensuremath{\lVert#1\rVert_{f,R_b}}}
\newcommand\normfRbone[1]{\ensuremath{\lVert#1\rVert_{f, R_{b_1}}}}
\newcommand\normfRbtwo[1]{\ensuremath{\lVert#1\rVert_{f,R_{b_2}}}}
\newcommand\normtwo[1]{\ensuremath{\lVert#1\rVert_{2}}}
\newcommand\norminfty[1]{\ensuremath{\lVert#1\rVert_{\infty}}}

\newcommand{\ds}{\displaystyle}

\allowdisplaybreaks

\title{Transition Fronts of Fisher-KPP Equations in Locally Spatially Inhomogeneous Patchy Environments I: Existence and Non-existence}

\author{
Erik S. Van Vleck\thanks{Erik Van Vleck Email: erikvv@ku.edu }\\
Department of Mathematics\\
University of Kansas\\
Lawrence, KS 66045, U.S.A. \\
\\
and \\
\\
Aijun Zhang\thanks{CORRESPONDENCE AUTHOR Aijun Zhang Email: zhangai@tigermail.auburn.edu.}\\
Department of Mathematics\\
University of Louisiana at Lafayette\\
Lafayette, LA 70504, U.S.A.}

\date{}
\maketitle

\noindent {\bf Abstract.}
This paper is devoted to the study of spatial propagation dynamics of species in locally spatially inhomogeneous patchy environments or media. For a lattice differential equation with monostable nonlinearity in a discrete homogeneous media, it is well-known that there exists a minimal wave speed such that a traveling front exists if and only if the wave speed is not slower than this minimal wave speed. We shall show that strongly localized spatial inhomogeneous patchy environments may prevent the existence of transition fronts (generalized traveling fronts). Transition fronts may exist in weakly localized spatial inhomogeneous patchy environments but only in a finite range of speeds, which implies that it is plausible to obtain a maximal wave speed of  existence of transition fronts.

\bigskip


\bigskip


\section{Introduction}
\setcounter{equation}{0}
Front propagation occurs in many applied fields such as population dispersals in biology, combustion in chemistry, neuronal waves in neuroscience, fluid dynamics in physics and more. Since the pioneering work of Fisher (\cite{Fisher}) and Kolmogorov-Petrovskii-Piskunov (\cite{KPP}), front propagation dynamics of classical reaction-diffusion equation

\vspace{-.1in}\begin{equation}
\label{RD-eq}
u_{t}(t,x)=u_{xx}+f(x,u)u, x\in\RR
\vspace{-.1in}\end{equation}
and lattice differential equation
\vspace{-.1in}\begin{equation}
\label{main-eq}
\dot{u}_{j}(t)=u_{j+1}-2u_{j}+u_{j-1}+f_j(u_{j})u_j,\quad j \in \ZZ.
\vspace{-.1in}\end{equation}
have been studied extensively. In biology \eqref{RD-eq} is used to model the spread of population in non-patchy environment with random internal interaction of the organisms and \eqref{main-eq} is for species in patchy environment with nonlocal internal interaction of the organisms. Here we focus on \eqref{main-eq}. For nonlinearity term $f_j(u_j)$, we assume that

\medskip
\noindent{\bf (H1)}  {\it $f_j\in C^2([0,\infty),\RR)$,
 $\ds-L<\inf_{j\in\ZZ,v\geq 0}\{f_j'(v)\}\leq\sup_{j\in\ZZ,v \geq 0}\{f_j'(v)\}<0$ for all $(j, v) \in \ZZ \times \RR^+$ with some $L>0$ and
 $f_j(v)<0$ for all $(j, v) \in \ZZ \times \RR^+$ with $v > L_0$ for some $L_0>0$.}
%
 
In the literature, (H1) is called Fisher-KPP type nonlinearity due to Fisher (\cite{Fisher}) and Kolmogorov-Petrovskii-Piskunov (\cite{KPP}).
However, most existing works are concerned with the propagation dynamics in homogeneous or spatially periodic media. Fisher (\cite{Fisher}) and Kolmogorov-Petrovskii-Piskunov (\cite{KPP}) considered a homogenous case of \eqref{RD-eq}, that is, $f(x,u)=f(u)=1-u$. Fisher conjectured and Kolmogorov-Petrovskii-Piskunov proved that there exist traveling fronts of speeds not less than the minimal wave speed $c^*=2$, which is a solution of \eqref{RD-eq} of form $u(t,x)=\phi(x-ct)$, $\phi(-\infty)=1$ and $\phi(\infty)=0$. Later, existence of periodic traveling waves of \eqref{RD-eq} or more general reaction diffusion equations with Fisher-KPP nonlinearity has been studied by researchers including B. Zinner and his collaborators in 1995 (\cite{HuZi1}), H.F.Weinberger in 2002 (\cite{Wei}), and H. Berestycki et al. in 2005 (\cite{BeHaNa1}). For the case in non-periodic inhomogeneous media, we can not expect wave profiles that take the form of constant or periodic front profiles. The notation of traveling waves has been extended to generalized traveling waves or transition fronts by several authors (e.g., \cite{BeHa},\cite{Shen}). In the past decade, transition fronts in non-periodic inhomogeneous media have attracted much attention (e.g., \cite{BeHa}, \cite{NoRoRyZl}, \cite{Zlatos}). For instance, J. Nolen et al. considered in \cite{NoRoRyZl} the KPP equation of one dimension with random dispersal (classic reaction-diffusion equation) in compactly supported inhomogeneous media. More precisely, they considered \eqref{RD-eq} in the media which are localized perturbations of the homogeneous media. They showed that localized KPP inhomogeneity may prevent the existence of transition fronts and provided some examples that transition fronts may not exist.

The discrete system \eqref{main-eq} has also been the subject of much research attention. The past two decades have seen vigorous research activities on applications to dynamics on lattice differential equations \cite{CMPVV99,ChenGuo,ChenGuo1,CGW,ChowMPShen,GuoWu,HHVV15,HHVV17}. In numerical simulations, lattice differential equations have some advantages over classical reaction-diffusion equations in applications. For example, \eqref{main-eq} can be viewed as the spatial discretization of \eqref{RD-eq}. On the other hand, lattice differential equations are of interest as models in their own right. It is more reasonable to model some problems with spatial discrete structure such as population dispersal in a patchy environment by lattice differential equations. The main concerns include also the properties of spreading speed and propagation of waves such as traveling fronts, periodic(pulsating) traveling waves and transition fronts. For homogeneous or periodic discrete media with monostable or bistable nonlinearities, we refer the readers to \cite{CMPVV99,ChenGuo,ChenGuo1,CGW,HHVV15,HHVV17}. The simplest case of transition fronts are traveling waves whose profiles are time-independent, that is, there exists some function $\phi$ such that
\vspace{-.1in}\begin{equation}
\label{tws}
u_{j}(t)=\phi(j-ct), \phi(\infty)=0 \ \text{and} \  \phi(-\infty)=1,
\vspace{-.1in}\end{equation}
where $c$ is the wave speed. For the homogenous case with $f_j(u_{j})=1-u_j$, it is almost trivial that there exists a minimal wave speed $c^*$ such that a traveling wave exists if and only if the wave speed $c \geq c^*$. Later, the periodic traveling wave solutions have been investigated in \cite{GuHa,HudsonZinner} for the Fisher-KPP equation in periodically inhomogeneous media, where the periodic traveling wave solutions $u_{j}(t)$ to lattice differential equations such as \eqref{main-eq} satisfy the following
\vspace{-.1in}\begin{equation}
\label{ptws}
u_{j}(t+p/c)=u_{j-p}(t), \lim_{j \to -\infty}u_{j}(t)=1 \ \text{and} \ \lim_{j \to \infty}u_{j}(t)=0 \ \text{locally in} \ t \in \RR.
\vspace{-.1in}\end{equation}
Work on entire solutions or transition fronts for bistable reaction-diffusion equations in discrete media includes \cite{HHVV17,HMVV11}. However, less is known to the spreading dynamics to \eqref{main-eq} with Fisher-KPP nonlinearity in non-periodic inhomogeneous media.

Kong and Shen considered in \cite{KoSh1} the KPP equations of higher dimension with nonlocal, random or discrete dispersal in localized perturbations of the homogeneous media and investigate in \cite{KoSh2} the KPP equations with nonlocal, random or discrete dispersal in localized perturbations of the periodic media. They showed that the localized spatial inhomogeneity of the medium preserve the spatial spreading in all the directions. The lower bound of mean wave speed of \eqref{main-eq} can be obtained due to the spreading properties proved in \cite{KoSh2} and in \cite{KoSh1} for the particular case in localized perturbations of the homogeneous media. However, the existence and (general) non-existence of transition fronts have not yet been investigated for discrete dispersals.

We will focus on the study of existence and non-existence of transition fronts of \eqref{main-eq} with Fisher-KPP type nonlinearity in localized perturbations of spatially homogeneous patchy environments or media. Hereafter, we assume the following:

 \medskip

\noindent{\bf (H2)}  {\it $f_j(0)>0$ for all $j$ and $f_j(0)=1$
 for any $|j|>N$ with some positive integer $N$.}

 \medskip

Throughout the paper, we assume (H1)-(H2). Let $\Lambda : \mathcal{D}(\Lambda) \subset X \to X$ be defined by
\vspace{-.1in}\begin{equation}
\label{mainoperator}
 (\Lambda u)_{j}:= u_{j+1}-2u_{j}+u_{j-1}+ f_j(0)u_{j}, \forall u \in X,
\vspace{-.1in}\end{equation}
where $X=\{u | |u_j| < L, \text{for some $L>0$ and all $j \in \ZZ$}\}$ with norm $\|u\|_{X}=\ds\sup_{j\in \ZZ} \{|u_j|\}$.

Let $\lambda=\sup\big\{\mathbf{Re}\ \mu\  |\ \mu \in\sigma(\Lambda)\big\}$. Let $\{u_j^*\}_{j \in \ZZ}$ be the unique positive stationary solution of \eqref{main-eq}, where the existence of $\{u_j^*\}_{j \in \ZZ}$ was proved in Theorem 2.1 of \cite{KoSh1} by Kong and Shen under the assumptions of (H1) and (H2). To study the propagation wave solutions in localized perturbations in patchy media, we will extend the traveling front of \eqref{tws} in homogeneous media and the periodic traveling front of \eqref{ptws} in periodic media and define transition fronts of \eqref{main-eq} and their mean speeds as follows:

 \begin{definition}[Transition Front]
\label{Transition-Fronts-def}
 $\{u_{j}(t)\}_{j \in \ZZ}$ is called a transition front of \eqref{main-eq} if it is an entire solution such that $0\leq u_{j}(t)\leq u_j^*$,
 $\ds\lim_{j \to -\infty} (u_{j}(t)-u_j^*)=0$ and $\ds\lim_{j \to \infty} u_{j}(t)=0$;
\end{definition}

 \begin{definition}[Mean Wave Speed]
 \label{meanspeed-def}
The value
$c$ is called the mean wave speed of the transition front given by $c=\ds\lim_{|t_j-t_k| \to \infty}\frac{j-k}{t_{j}-t_{k}}$, where $t_{i}$ is the time such that $\ds u_{i}(t_{i})=\frac{1}{2}\min_{j} \{u_j^*\}$ for $i \in \ZZ$ and $\ds u_{l}(t_{i})< \frac{1}{2}\min_{j} \{u_j^*\}$ for all $l>i$.
\end{definition}
In the current study, our main result shows conditions for both existence and nonexistence of transition fronts of \eqref{main-eq} for lattice differential KPP equation in patchy environment with a localized perturbation in media. There are several essential difference between classic reaction differential equations and lattice differential equations. Among these fundamental techniques are heat kernel estimate, Poincar$\acute{e}$ inequality, Harnack inequality and principal eigenvalue theory. We shall introduce discrete versions of these fundamental tools in later sections. Because of those significant differences, the approaches for classical reaction diffusion equations in \cite{NoRoRyZl} can not be applied directly to \eqref{main-eq}, that is a continuous-time discrete in space lattice differential equation. 
In this paper, we consider transition fronts in the localized perturbed homogeneous patchy media, and provide the variational formulas for both the upper bound and the lower bound of the wave speeds that transition fronts exist. 

Throughout the rest of paper, let $\lambda(\mu)=e^{\mu}-1+e^{-\mu}$ for $\mu > 0$. We have an auxiliary function for the wave speed, $c(\mu)=\frac{\lambda(\mu)}{\mu}$ for $\mu>0$. Let  $(c^*, \mu^*)$ be such that $c^*=\ds\frac{\lambda(\mu^*)}{\mu^*}=\inf_{\mu>0}\frac{\lambda(\mu)}{\mu}$. In literature, $c^*$ is so called spreading speed, that is the minimal speed such that a traveling solution may exist. We explore the minimal speed $c^*$ in Section \ref{lower bound section}. Let $\lambda^{*}=\lambda(\mu^*)$ and $(\hat{c}, \hat{\mu})$ be such that $\lambda=\lambda(\hat{\mu})$ and $\hat{c}=c(\hat{\mu})$. The $\hat{c}$ is corresponding to the maximal speed such that a traveling solution may exist (see Section \ref{upper bound section}). We state the main theorem in the following.

\begin{theorem}[Existence and Non-Existence of Transition Fronts]
\label{Existence and Nonexistence}
Assume (H1)-(H2).
\begin{itemize}
\item[(1)] If $\lambda \in [1,\lambda^{*})$ and $\hat{c}>c^{*}$, then transition front exists for any speed $c \in [c^{*}, \hat{c}]$. Moreover, if $c \in (c^*,\hat c]$, then for any $\epsilon>0$, there exist $C_1,\ C_2,\ T>0$ such that for $t>T$ and $j>ct$,
\begin{equation}
\label{tail_estimate_t0}
C_1e^{-(\mu+\epsilon)(j-ct)} \leq u_j(t) \leq C_2e^{-(\mu-\epsilon)(j-ct)}.
\end{equation}
\item[(2)] No transition front with speed c exists for the following cases: (i) $\lambda >\lambda^{*}$; (ii)$c<c^{*}$ and (iii) $c > \hat{c}$.
\end{itemize}
\end{theorem}

This paper is organized as follows. In Section 2, we provide the discrete analogs of fundamental tools in classical reaction diffusion equations, including semigroup theory, comparison principles, discrete heat kernel, discrete parabolic Harnack inequality and many others. In Section 3, we investigate the principal eigenvalue theory and construct the super/sub-solutions. Then we show the existence of transition fronts and also the tail estimates of transition fronts \eqref{tail_estimate_t0}, that is, proof of Theorem \ref{Existence and Nonexistence} (1). In Section 4, we show nonexistence of transition fronts under  $\lambda >\lambda^{*}$, the lower bound of wave speeds (minimal wave speed $c^*$), and the upper bound of wave speeds (maximal wave speed $\hat c$), that is proof of Theorem \ref{Existence and Nonexistence} (2). In Section 5, we provide a particular example with the simplest case: a perturbation at a single location. Finally, we provide some concluding remarks in Section 6.

\section{Foundations of Lattice Differential Equations }
\subsection{Initial Value Problem}
 Let $X^+=\{u\in X| u_j \geq 0, \forall j \in \ZZ\}$. Let $\Lambda$ be as in \eqref{mainoperator}. It follows from the general semigroup approach (see \cite{Paz}) that $\Lambda$ generates a uniformly continuous semigroup $T(t)$ and \eqref{main-eq} has a unique (local) solution $u(t;z)$ with $u(0)=\{z_{j}\}_{j \in \ZZ}$ for every $z\in X$, that is given by
\vspace{-.1in}\begin{equation}
\label{IVP-solution}
u(t)=T(t)u(0)-\int_{0}^{t}T(t-s)g(s)ds,t>0,
\vspace{-.1in}\end{equation}
where $g_j(s)=(f_j(u_j)-f_j(0))u_j$ for $j \in \ZZ$, $g(s)=\{g_{j}(s)\}_{j \in \ZZ}$, and $u(t)=\{u_{j}(t)\}_{j \in \ZZ}$.

\subsection{Comparison Principle}
We introduce comparison principle in this subsection, which will play an important role in obtaining the existence of transition fronts of \eqref{main-eq}. We define super/sub-solutions and state the comparison principle as follows.
\begin{definition}[Super/Sub-Solution]
\label{superSub-def}
For a given continuous-time and bounded function $u_{j} : [0, T ) \to \RR$, $\{u_j\}_{j \in \ZZ}$ is called a
super-solution (sub-solution) of \eqref{main-eq} on $[0,T)$ if for all $j$,
$\dot u_{j}(t) \geq (\leq)u_{j+1}-2u_{j}+u_{j-1} + f_{j}(u_{j})u_j$.
\end{definition}
\begin{proposition}[Comparison Principle]
\label{monotonicity-new-prop1} $\quad$
\begin{itemize}
\item[(1)]
If $u(t)$ and $v(t)$ are sub-solution and super-solution of \eqref{main-eq} on $[0,T)$,
respectively, $u_j(0)\leq v_j(0)$, then
$$
u_j(t)\leq v_j(t)\quad {\rm for}\quad t\in [0,T).
$$
Moreover, if $u_j(0)\neq v_j(0)$ for some $j$, then for all $j$,
$$
u_j(t)< v_j(t)\quad {\rm for}\quad t\in (0,T).
$$
\item[(2)] If  $z, w \in X$ and $z \leq w$, then $u_j(t;z)\leq u_j(t;w)$ for $t>0$
at which both $u(t;z)$ and $u(t;w)$ exist. Moreover, if $z_{j} \neq w_{j}$ for some $j$, then for all $j$, $u_j(t;z)< u_j(t;w)$ for $t>0$
 at which both $u(t;z)$ and $u(t;w)$ exist.
\end{itemize}
\end{proposition}
\begin{proof}
The proof follows from arguments in Lemma 2.1 in \cite{CGW}.
\end{proof}
With the comparison principle, we have that if $z \in X^+$, $u(t;z) \in X^+$.
\medskip

In next two subsections, we introduce the discrete heat kernel and the discrete parabolic Harnack inequality, which play critical roles in studying the tail estimates and the bounds of wave speeds of transition fronts. 

\subsection{Discrete Heat Kernel}
Discrete heat kernel is highly related to I-Bessel functions. The I-Bessel function $I_{x}(t)$ is defined as a solution to the differential equation
$$t^{2} \frac{d^{2}y}{dt^{2}}+ t\frac{dy}{dt}- (t^2 + x^2) = 0.$$
In \cite{Pal99}, the author derived an upper bound and lower bound for $I_{x}(t)$, for all $t > 0$ and $x \geq 0$,
$$e^{-\frac{1}{2\sqrt{t^{2}+x^{2}}}}\leq I_{x}(t) \sqrt{2\pi}(t^2 + x^2)^{\frac{1}{4}}e^{-\varsigma_0(t,x)}
\leq e^{\frac{1}{2\sqrt{t^{2}+x^{2}}}},$$ with $\varsigma_0(x, t)=\sqrt{t^{2}+x^{2}}+xln(\frac{t}{x+\sqrt{t^{2}+x^{2}}})$.

By Proposition 3.1 in \cite{CJK2}, the heat kernel on a 2-regular graph is given by $$K(t,r)=e^{-2t}I_{r}(2t), \text{ for } (t ,r)  \in (0, \infty) \times \ZZ^{+}.$$ With the help of the above bounds of $I_{r}(t)$, we have the bounds of $K(t,r)$:
$$\frac{1}{\sqrt{2\pi}}(4t^2 + r^2)^{-\frac{1}{4}}e^{-2t-\frac{1}{2\sqrt{4t^{2}+r^{2}}}+\varsigma_0(2t,r)}\leq K(t,r)\leq \frac{1}{\sqrt{2\pi}}(4t^2 + r^2)^{-\frac{1}{4}}e^{-2t+\frac{1}{2\sqrt{4t^{2}+r^{2}}}+\varsigma_0(2t,r)}.$$ The authors in \cite{CJK2} showed that $\sqrt{t}e^{-t}I_{x}(t)\leq (1+\frac{x}{t})^{-\frac{x}{2}}$, thus $K(t,r)\leq \frac{1}{\sqrt{2t}}(1+\frac{r}{2t})^{-\frac{r}{2}}$.

By Theorem 2.3 in \cite{CMS}, $h^{\ZZ}_{t}(j) \asymp F(t,j)$, that is, there exist positive real constants $\epsilon > 0$ and $M_{\epsilon}>0$ such that
\vspace{-.1in}\begin{equation}
\label{heatkernel}
(1-\epsilon)F(t,j) \leq h^{\ZZ}_{t}(j)\leq  (1+\epsilon)F(t,j),
\vspace{-.1in}\end{equation}
for $j^2+t^2> M_{\epsilon}$, where $h^{\ZZ}_{t}(j)$ is the heat kernel associated with $\mathcal{L} f(j)=f(j) - \frac{f(j+1)+f(j-1)}{2} $ and $F(t,j)$ is given by if $j=0$,$$F(t,j)=\frac{1}{\sqrt{2\pi}}\frac{1}{(1+t^2)^{\frac{1}{4}}},$$ else if $j\neq 0$,

 $$F(t,j)=\frac{1}{\sqrt{2\pi}}\frac{exp[-t+|j|\varsigma(t/|j|)]}{(1+t^2+j^2)^{\frac{1}{4}}},$$
 where $\varsigma(t/|j|):=\varsigma_0(1,t/|j|)$. 

Recall the nonlinear equation \eqref{main-eq},
\[
\ds\dot{u}_{j}=u_{j+1}-2u_{j}+u_{j-1}+f_j(u_j) u_{j},\quad j \in \ZZ.
\]
Consider also the linearized equation
\vspace{-.1in}\begin{equation}
\label{linear-eq}
\ds\dot{u}_{j}=u_{j+1}-2u_{j}+u_{j-1}+a_j u_{j},\quad j \in \ZZ,
\vspace{-.1in}\end{equation}
where $a_j=f_j(0)$

Let $\Lambda_s : \mathcal{D}(\Lambda_s) \subset X \to X$ be defined by
\vspace{-.1in}\begin{equation}
\label{mainoperator_0}
 (\Lambda_s u)_{j}:= u_{j+1}-u_{j}+u_{j-1}, \forall u \in X.
\vspace{-.1in}\end{equation}

Let $S(t)$ be the semigroup generated by $\Lambda_s$. Note that $(S(t) z)_j=\ds e^t \sum_k h^{\ZZ}_{2t}(j-k) z_k$ for $z := \{z_{j}\}_{j \in \ZZ} \in X$. Then the solution of \eqref{main-eq} is given by $$u(t)=S(t-T)u(T)-\int_{T}^{t}S(t-s)g(s)ds,t>T,$$
where $g_j(t)=(1-f_j(u_j))u_j(t)$. More precisely, we have the following, for $t>T$,
\vspace{-.1in}\begin{equation}
\label{nonlinear-sol}
u_j(t)=e^{t-T}\sum_kh^{\ZZ}_{2(t-T)}(j-k)u_k(T)-\int_{T}^{t}e^{(t-s)}\sum_kh^{\ZZ}_{2(t-s)}(j-k))g_k(s)ds.
\vspace{-.1in}\end{equation}
We should point out that the solution form with \eqref{nonlinear-sol} is slightly different with that given by \eqref{IVP-solution}. With heat kernel $h^{\ZZ}_{2t}$ in \eqref{nonlinear-sol}, we can use the heat kernel estimate \eqref{heatkernel}. Then there would be some advantages over \eqref{IVP-solution} while exploring some estimates, such as the exponential tail estimates of transition fronts.

\subsection{Discrete Parabolic Harnack Inequality}
In this subsection, we shall introduce the discrete parabolic Harnack inequality for the solution to our main equation \eqref{main-eq}. Harnack inequalities have many significant applications in both elliptic and parabolic differential equations such as exploring boundary regularity, heat kernel estimate, and other solution estimates. Moser in \cite{Moser} proved a parabolic Harnack inequality for classical parabolic PDEs. For discrete parabolic Harnack inequalities, we will adopt Definition 1.6 and apply Theorem 1.7 in \cite{Delmotte} to prove that the discrete parabolic Harnack inequality holds on a 2-regular graph. Readers are referred to \cite{Delmotte} for further information about parabolic Harnack inequality on graphs. For convenience, we recall necessary graph theory, and state the Definition 1.6 of \cite{Delmotte} as the following Definition \ref{Harnack_Inequality_2_graph}.

Let $\Gamma$ be an infinite set and $\mu_{xy} =\mu_{yx}$ a symmetric nonnegative weight on $\Gamma \times \Gamma$. We call $x$ and $y$ neighbors, denoted by $x \sim y$, when $\mu_{xy}\neq 0$. Vertices are measured by $m(x) = \ds\sum_{x \sim y} \mu_{xy}$. The “volume” of subsets $E \subset \Gamma$ by $V(E)= \ds\sum_{x \in E} m(x)$. We can further define $d (x, y)$ as the distance of $x$ and $y$ in $\Gamma$, that is, the shortest number of edges between $x$ and $y$. Let $B_r(x)$ be the closed ball $\{y \in \Gamma | d(x,y) \leq r\}$. We say that $u(t,x)$ satisfies continuous-time parabolic equation on $(t,x)$ if 
\vspace{-.1in}\begin{equation}
\label{Parabolic-graph}
m(x) u_t(t,x)=\sum_{y} \mu_{xy}(u(t,y)-u(t,x)).
\vspace{-.1in}\end{equation}

We remark that for a 2-regular graph, $x$ has only two neighbors $y_-:=x-1$ and $y_+:=x+1$. If we consider the same weight for $\mu_{xy_-}=\mu_{xy_+}$, then $$\sum_{y} \mu_{xy}(u(t,y)-u(t,x))= \mu_{xy_-} (u(t,x-1)-2u(t,x)+u(t,x+1)),$$ that is the exactly same type equation as \eqref{main-eq} we consider in the paper. In \cite{Delmotte}, Delmotte defines Harnack inequality of \eqref{Parabolic-graph} on the graph as follows. 

\begin{definition}[Harnack Inequality \cite{Delmotte}]
\label{Harnack_Inequality_2_graph}
Set $\eta \in(0,1)$ and $0<\theta_1<\theta_2<\theta_3<\theta_4$. $(\Gamma,\mu)$ satisfies the continuous-time parabolic Harnack inequality $H(\eta,\theta_1,\theta_2,\theta_3,\theta_4,C)$ if for all $x_0,s, r$ and every nonnegative solution on
$Q=[s,s+\theta_4 r^2] \times B_r(x_0)$ we have
$$
\sup_{Q_-} u \leq C \inf_{Q_+} u,
$$
where $Q_-=[s+\theta_1 r^2,s+\theta_2 r^2] \times B_{\eta r}(x_0)$ and $Q_+=[s+\theta_3 r^2,s+\theta_4 r^2] \times B_{\eta r}(x_0)$.
 \end{definition}
By Theorem 1.7 in \cite{Delmotte}, the discrete parabolic Harnack inequality holds if and only if the following three conditions are satisfied:
 \begin{definition}[$\Delta^*( \alpha )$ Condition]
Let $\alpha>0$, the weighted graph satisfies $\Delta^*  ( \alpha )$ if $$x \sim y \implies \mu_{xy} \geq \alpha m(x);$$
 \end{definition}
  \begin{definition}[$''$Doubling Volume$''$ Property]
There exists a  $C>0$ such that $$V(B_{2r}(x)) \leq C V(B_{r}(x))$$ for any $x \in \Gamma$ and $r$;
 \end{definition}
 and
\begin{definition}[Poincar$\acute{e}$ Inequality]
There exists a $C_2>0$ such that for all $v \in \RR^{ \Gamma}$, all $x_0 $, and $r>0$, 
$$
 \sum_{x \in B_r(x_0)}m(x) (v(x)-\bar v)^2 \leq C_2 r^2 \sum_{x,y  \in B_{2r}(x_0)} \mu_{xy}(v(x)-v(y))^2,
$$
where $\bar v=\ds\frac{1}{V(B_r(x_0))}\sum_{x  \in B_r(x_0)}m(x) v(x)$.
 \end{definition}
 Now we claim that parabolic Harnack inequality holds on a 2-regular graph.
 \begin{theorem}[Harnack Inequality on a 2-regular Graph]
\label{Harnack_inequality_2graph }
The parabolic Harnack inequality $H(\eta,\theta_1,\theta_2,\theta_3,\theta_4,C)$ holds on a 2-regular graph.
\end{theorem}
\begin{proof}
It suffices to show a 2-regular graph satisfies the $\Delta^*(\alpha)$ condition, the $''$doubling volume$''$ property and the Poincar$\acute{e}$ inequality.
First, a 2-regular graph with $0< \alpha \leq \frac{1}{2}$ satisfies the $\Delta^*  ( \alpha )$ condition.
Second, for a 2-regular graph, $V(B_{r}(x))=2(2r+1)$ and $V(B_{2r}(x))=2(4r+1)$. Choose $C=2$ and then the $''$doubling volume$''$ property holds.

Finally, we prove the Poincar$\acute{e}$ inequality on a 2-regular graph. In fact, we have a strong Poincar$\acute{e}$ inequality, that is, $B_{2r}(x_0)$ can be reduced by $B_{r}(x_0)$. Without loss of generality, let $x_0=0$ and consider $v(x)$ for $-r \leq x \leq r $. Consider the same weights for all vertices, and let $\mu_{xy}=\mu_{xy}=1$ if $|y-x|=1$, otherwise 0. Then we have 
\vspace{-.1in}\begin{equation}
\label{Poincare-eq}
 \sum_{x,y  \in B_{r}(x_0)} \mu_{xy}(v(x)-v(y))^2= \sum_{x  \in B_{r}(x_0)} [(v(x)-v(x+1))^2+(v(x)-v(x-1))^2].
\vspace{-.1in}\end{equation}
 The sequence $v(x)$ oscillates around $\bar v$. In other words, if $v(x)$ moves from $-r$ to $r$, it must either hit the $\bar v$ at some point or cross $\bar v$ from one side to another. There exists at least one integer $\hat x$ such that either $v(\hat x)=\bar v$ or $(v(\hat x)-\bar v)(v(\hat x+1)-\bar v)<0$. In addition, there exists an $\hat x \in (-r, r)$ such that
\vspace{-.1in}\begin{equation}
\label{Oscillation-eq}
\max\{|v(\hat x+1)-\bar v|,|v(\hat x)-\bar v|\} \leq |v(\hat x)-v(\hat x+1)|.
\vspace{-.1in}\end{equation}
Thus, with \eqref{Poincare-eq}, \eqref{Oscillation-eq}, Cauchy-Schwarz and triangle inequalities, we have that, for $x \leq \hat x$,
\begin{align*}
| v(x)-\bar v|&=|\sum_{y = x}^{\hat x-1}(v(y)-v(y+1))+(v(\hat x)-\bar v)|\\
 &\leq\sum_{y = x}^{\hat x-1}|(v(y)-v(y+1))|+|(v(\hat x)-\bar v)|\\
  &\leq\sum_{y = x}^{\hat x-1}|(v(y)-v(y+1))|+|(v(\hat x)-v(\hat x+1))|\\
  &= \sum_{y = x}^{\hat x}|(v(y)-v(y+1))|\\
  &\leq \sum_{y = -r}^{r-1}|(v(y)-v(y+1))|\\
  &\leq [\sum_{y = -r}^{r-1}((v(y)-v(y+1)))^2]^{\frac{1}{2}}[\sum_{y = -r}^{r-1}(1)^2]^{\frac{1}{2}}\\
  &= [2r \sum_{y = -r}^{r-1}((v(y)-v(y+1)))^2]^{\frac{1}{2}}\\
  &\leq [2r  \sum_{x,y  \in B_{r}(x_0)} \mu_{xy}(v(x)-v(y))^2]^{\frac{1}{2}}.
\end{align*}
If $x=\hat x+1$, with \eqref{Oscillation-eq}, $|v(\hat x+1)-\bar v| \leq |v(\hat x)-v(\hat x+1)|$ and so we also have the above inequality.
If $x>\hat x+1$, then we can do backward arguments above and have
\begin{align*}
| v(x)-\bar v|&=|\sum_{y =\hat x+2}^{x}(v(y)-v(y-1))+(v(\hat x+1)-\bar v)|\\
 &\leq \sum_{y =\hat x+2}^{x}|(v(y)-v(y-1))|+|(v(\hat x+1)-\bar v)|\\
  &\leq\sum_{y =\hat x+2}^{x}|(v(y)-v(y-1))|+|(v(\hat x)-v(\hat x+1))|\\
  &= \sum_{y =\hat x+1}^{x}|(v(y)-v(y-1))|\\
  &\leq \sum_{y = -r+1}^{r}|(v(y)-v(y-1))|\\
  &\leq [\sum_{y = -r+1}^{r}((v(y)-v(y+1)))^2]^{\frac{1}{2}}[\sum_{y = -r+1}^{r}(1)^2]^{\frac{1}{2}}\\
  &=[2r\sum_{y = -r+1}^{r}((v(y)-v(y+1)))^2]^{\frac{1}{2}}\\
  &\leq [2r  \sum_{x,y  \in B_{r}(x_0)} \mu_{xy}(v(x)-v(y))^2]^{\frac{1}{2}}.
\end{align*}
In summary, for all $x \in B_{r}(x_0)$, we have that 
$$| v(x)-\bar v| \leq  [2r  \sum_{x,y  \in B_{r}(x_0)} \mu_{xy}(v(x)-v(y))^2]^{\frac{1}{2}}.$$
Take the square for both sides and thus
$$ (v(x)-\bar v)^2 \leq  2r  \sum_{x,y  \in B_{r}(x_0)} \mu_{xy}(v(x)-v(y))^2.$$
Note that $m(x)= \ds\sum_{x \sim y} \mu_{xy}=2$. Then take the sum over $B_r(x_0)$ and we have
\begin{align*}
\sum_{x \in B_r(x_0)}m(x) (v(x)-\bar v)^2 &\leq   (m(x) (2r+1)2r) \sum_{x,y  \in B_{r}(x_0)} \mu_{xy}(v(x)-v(y))^2\\
&=   (8r^2+4) \sum_{x,y  \in B_{r}(x_0)} \mu_{xy}(v(x)-v(y))^2\\
&\leq   (12 r^2) \sum_{x,y  \in B_{r}(x_0)} \mu_{xy}(v(x)-v(y))^2.
\end{align*}
Hence, Poincar$\acute{e}$ inequality holds for $C_2=12$.
 \end{proof}
 \subsection{Some Auxiliary Functions}
 We recall some auxiliary functions. One is for the function $\varsigma(z)$ in the heat kernel estimate \eqref{heatkernel}. Recall that $\varsigma(z)=\sqrt{1+z^2}+ln\frac{z}{1+\sqrt{1+z^2}}$ for $z \in \RR^+$.  Another is for the wave speed, $c(\mu)=\frac{\lambda(\mu)}{\mu}$ with $\lambda(\mu)=e^{\mu}-1+e^{-\mu}$ for $\mu > 0$. The properties of these auxiliary functions play important roles throughout later sections. We group them in the following lemma and their proofs are straightforward.
 \begin{lemma}
\label{Au-lm4}
\begin{itemize} Let $g(z)=-1+2\frac{\varsigma(z)+\mu}{z}$ for $\mu>0$ and $z>0$.
\item[(1)] $\varsigma(z)$ is strictly increasing in $z$ on $(0,\infty)$ and then there exists a $l_0>0$ such that $\varsigma(l_0)=0$.
\item[(2)] $g(z)$ is concave down and obtains an absolute maximum at $z_0=csch(\mu)=\frac{2}{e^{\mu}-e^{-\mu}}$ for $z \in (0,\infty)$ and $g(z_0)=\lambda(\mu)$. 
\item[(3)] For fixed $\mu>0$, $c(\mu)$ is concave up and has a unique critical point at $\mu^*$, that is, $c(\mu)$ strictly decreasing in $(0,\mu^*]$ and strictly increasing in $(\mu^*,\infty)$.
\item[(4)] For $\mu \in (0,\mu^*)$, $c(\mu)>\frac{2}{z_0}$, for $\mu = \mu^*$, $c(\mu)=\frac{2}{z_0}$ and for $\mu >\mu^*$, $c(\mu)<\frac{2}{z_0}$, where $z_0=csch(\mu)$.
\item[(5)] $\frac{\varsigma(z)}{z}$ is strictly increasing in $z$ on $(0,\infty)$ and $\ds\lim_{z \to \infty}\frac{\varsigma(z)}{z}=1$.
\end{itemize}
\end{lemma}
\begin{proof}
\begin{itemize}
\item[(1)] By direct computation,
$$\varsigma'(z)=\frac{z}{1+\sqrt{1+z^2}}+\frac{1}{z}=\frac{\sqrt{1+z^2}}{z}>0.$$
Therefore $\varsigma(z)$ is strictly increasing on $(0,\infty)$. Since $\varsigma(z) \to -\infty$ as $z \to 0$  and $\varsigma(z) \to \infty$ as $z \to \infty$, there exists a $l_0>0$ such that $\varsigma(l_0)=0$.
\item[(2)] By direct computation, $g'(z)=2\frac{z\varsigma'(z)-\varsigma(z)-\mu}{z^2}=2\frac{ln\frac{1+\sqrt{1+z^2}}{z}-\mu}{z^2}$ for $\mu>0$. Then there exists a unique critical point $z_0=\frac{2}{e^{\mu}-e^{-\mu}}$ such that $g'(z_0)=0$. We can verify that $g(z)$ obtains an absolute maximum at $z_0$ by first derivative test. Since $ln\frac{1+\sqrt{1+z^2}}{z}$ is a strictly decreasing function with the range from positive infinity to 0, $g'(z)>0$ for $z<z_0$ and $g'(z)<0$ for $z>z_0$.
Plugging $z_0$ into $\frac{\varsigma(z)+\mu}{z}$,
\begin{align*}
\frac{\varsigma(z_0)+\mu}{z_0}&=\frac{\varsigma(z_0)}{z_0}+\frac{\mu}{z_0}\\
&=\frac{\varsigma(csch(\mu))}{csch(\mu)}+\frac{\mu}{csch(\mu)}\\
&=\frac{\sqrt{1+csch^{2}(\mu)}+ln\frac{csch(\mu)}{1+\sqrt{1+csch^{2}(\mu)}}}{csch(\mu)}+\frac{\mu}{csch(\mu)}\\
&=\frac{\sqrt{coth^{2}(\mu)}+ln\frac{csch(\mu)}{1+\sqrt{coth^{2}(\mu)}}}{csch(\mu)}+\frac{\mu}{csch(\mu)}\\
&=\frac{coth(\mu)+ln\frac{csch(\mu)}{1+coth(\mu)}}{csch(\mu)}+\frac{\mu}{csch(\mu)}\\
&=\frac{coth(\mu)+ln\frac{1}{sinh(\mu)+cosh(\mu)}}{csch(\mu)}+\frac{\mu}{csch(\mu)}\\
&=\frac{coth(\mu)-\mu}{csch(\mu)}+\frac{\mu}{csch(\mu)}\\
&=cosh(\mu)
\end{align*}

Thus, $g(z_0)=-1+2cosh(\mu)=e^{\mu}+e^{-\mu}-1=\lambda(\mu)$.
\item[(3)] We can prove it by direct computation of solving $c'(\mu)=0$ and verifying $c''(\mu)>0$.
\item[(4)] Let $h(\mu)=c(\mu)-\frac{2}{z_0}$. Then $h(\mu)=\frac{\lambda(\mu)}{\mu}-(e^{\mu}-e^{-\mu}))=\frac{\lambda(\mu)-\mu(e^{\mu}-e^{-\mu})}{\mu}=-\mu c'(\mu)$ and so $h(\mu)$ has an opposite sign as $c'(\mu)$. By (3), $c'(\mu)<0$ for $\mu \in (0,\mu^*)$, $c'(\mu^*)=0$ and $c'(\mu)>0$ for $\mu>\mu^*$, as required.
\item[(5)] Since $(\frac{\varsigma(z)}{z})'=\frac{\varsigma'(z)z-\varsigma(z)}{z^2}=-\frac{1}{z^2}ln\frac{z}{1+\sqrt{1+z^2}}>0$, $\frac{\varsigma(z)}{z}$ is a strictly increasing function on $(0,\infty)$. The limit $\ds\lim_{z \to \infty}\frac{\varsigma(z)}{z}=1$ follows easily.
\end{itemize}
\end{proof}

\section{Existence of Transition Fronts and Tail Estimates}
This section is devoted to investigating the existence of transition fronts of \eqref{main-eq} for wave speed $c \in [c^*, \hat c]$ when $\lambda \in [1, \lambda^*)$ and $c^* < \hat c$. By Lemma \ref{Au-lm4}(3), the wave speed interval $c \in [c^*, \hat c]$ corresponds to the interval of $\mu \in [\hat \mu, \mu^*]$. To prove the existence of transition fronts, we apply fundamental tools such as comparison principles and constructions of super- and sub-solutions. First we introduce principal eigenvalue theory for Jacobi operators, that will play a central and important role in these processes. 
\subsection{Principal Eigenvalue Theory for Jacobi Operators}
Let $X_{\mu}=\{u \in X:  \sup |e^{j\cdot \mu}u_j| < \infty \}$. Consider the following linear difference equation for $\psi \in X_{\mu}$,
\vspace{-.1in}\begin{equation}
\label{eigen-eq1}
\Lambda \psi_{j}=\gamma \psi_{j},
\vspace{-.1in}\end{equation}
where $j \in\ZZ$, $\Lambda$ is as in \eqref{mainoperator}.

Note that letting $\psi_{j}=e^{-\mu j}\phi_{j}$, we have the following equivalent problem for $\phi \in X$,
\vspace{-.1in}\begin{equation}
\label{eigen-eq}
e^{-\mu}\phi_{j+1}-2\phi_{j}+e^{\mu}\phi_{j-1}+a_{j}\phi_{j}=\gamma_\mu \phi_{j},
\vspace{-.1in}\end{equation}
where $j \in\ZZ$ with $a_j=f_j(0)$.

Let $\Lambda_{\mu}: \mathcal{D}(\Lambda_{\mu}) \subset X \to X$ be defined by
$(\Lambda_{\mu} \phi)_{j}:= e^{-\mu}\phi_{j+1}-2\phi_{j}+e^{\mu}\phi_{j-1}+a_{j}\phi_{j}.$ $\Lambda_{\mu}$ are of so called Jacobi operators in \cite{Teschl}. The positive principal eigenvectors to \eqref{eigen-eq} play important roles in constructions of transition fronts to \eqref{main-eq}. We refer readers to \cite{Teschl} for spectral theory of Jacobi operators in detail. For the particular case of periodic media, we refer readers to \cite{GuHa}. 

We can obtain that $\lambda$ is a principal eigenvalue for \eqref{eigen-eq1} by classical Krein-Rutman Theorem \cite{Krein-Rutman}, that is, a generalized version of Perron-Frobenius theorem. Sometimes we want to consider the truncated eigenvalue problem of \eqref{eigen-eq1}:
\vspace{-.1in}\begin{equation}
\label{truncated-eigen-eq}
\phi_{j+1}-2\phi_{j}+\phi_{j-1}+a_{j}\phi_{j}=\lambda_M \phi_{j},
\vspace{-.1in}\end{equation}
 where $j \in [-M,M]$, and $\phi_{M+1}=\phi_{-M-1}=0$ for $M>N$. If we write it in a matrix form, and let $A_M$ be
\[ \left( \begin{array}{ccccc}
a_{-M}-2 & 1 & 0&...& 0 \\
1 & a_{-M+1}-2 & 1 &...& 0\\
...&... & ... & ... &...\\
0&...&1 & a_{M-1}-2 & 1\\
0&...&0&1 & a_{M}-2  \end{array} \right),\]
then 
\vspace{-.1in}\begin{equation}
\label{eigen-eq-trucated}
A_M\phi^M=\lambda_M \phi^{(M)},
\vspace{-.1in}\end{equation}
 where $\phi^{(M)}=(\phi_{-M},...,\phi_{M})^T$. By Perron-Frobenius theorem, there exists a principal eigenvalue and an associated positive eigenvector. We let $(\lambda_M, \phi^{(M)}_j)$ be the pair of corresponding $l^\infty$ normalized principal eigenvalue and eigenvector, that is, $(\lambda_M, \phi^{(M)}_j)$ satisfies \eqref{eigen-eq} with $\|\phi^{(M)}\|_\infty=1$ and $\phi^{(M)}_j>0$ for $j \in [-M, M]$. Let $M$ go to infinity and we claim that the limit of $\lambda_{M}$ exists and is $\lambda$.  
\begin{lemma}
\label{lambda_limit-lemma}
 $\lambda \geq 1$ and $\ds\lambda=\lim_{M \to \infty}\lambda_{M}$ for $M>N$.
\end{lemma}
\begin{proof}
Clearly, $\lambda \geq 1$. Without loss of generality, let both $M$ and $N$ are even. We let $\tilde{A}_M$ be an $2M+3$ by $2M+3$ matrix:
\[ \left( \begin{array}{ccccc}
0 & 0 & 0&...& 0 \\
0 & a_{-M}-2 & 1 &...& 0\\
...&... & ... & ... &...\\
0&...&1 & a_{M}-2 & 0\\
0&...&0&0& 0  \end{array} \right).\]
Then we have that $A_{M+1} \ge  \tilde A_M$ and so $\rho (A_{M+1})\ge \rho (\tilde A_{M}) =\rho (A_{M})$, where $\rho (\#)$ is the spectral radius of the matrix $\#$. Thus, $\lambda_{M}=\rho (A_{M})$ is non-decreasing in $M$. On the other hand, $\|A_M\|_{max}=\ds\max_{i,j}|A_M(i,j)|$, where $A_M(i,j)$ is the element of $A_M$ at the i-th row and j-th column. Then $0< \lambda_{M}\le \|A_M\|_{max} \le \ds\max_{j}\{|a_j|+2\}$, that is, $\lambda_{M}$ is uniformly bounded. Therefore, the limit $\ds\lim_{M \to \infty}\lambda_{M}$ exists and it is denoted by $\lambda_{\infty}=\ds\lim_{M \to \infty}\lambda_{M}$. Let $\phi^{(M)}$ be the positive eigenvector of $A_M$ with $\|\phi^{(M)}\|_{\infty}=1$. For each $j$, there exists a subsequence $M_j$ of M such that  $\ds\lim_{M_j \to \infty}\phi_j^{(M_j)}$ exists and let $\phi_j^{(\infty)}=\ds\lim_{M_j \to \infty}\phi_j^{(M_j)}$. For each $M>>N$, let $j_M$ be such that $\phi_{j_M}^{(M)}=1$. For $j<-N$, we write \eqref{truncated-eigen-eq} as the following,
\vspace{-.1in}\begin{equation}
\label{eigen-eq_3}
\phi_{j+1}=(1+\lambda_{M})\phi_{j}-\phi_{j-1}.
\vspace{-.1in}\end{equation}
Let $c_1=1+\lambda_{M}$ and $c_2=-1$. We can solve a recursive sequence $\phi_{j+1}=c_1\phi_{j}+c_2\phi_{j-1}$. To this end, we use an auxiliary equation $x^2-c_1x-c_2=0$. Then solve it to have two roots $d_1=\frac{1+\lambda_{M} + \sqrt{(1+\lambda_{M})^2+4}}{2}>1$ and $d_2=\frac{1+\lambda_{M} - \sqrt{(1+\lambda_{M})^2+4}}{2}<0$. Therefore, we have
\begin{align*}
&\phi_{j+1}-d_1\phi_{j}=d_2(\phi_{j}-d_1\phi_{j-1}),\\
&\phi_{j+1}-d_2\phi_{j}=d_1(\phi_{j}-d_2\phi_{j-1}).
\end{align*}
Note that $\phi_{-M-1}=0$ and $\phi_{-M}>1$. Thus, for $-M \leq j<-N$, we have 
\begin{align*}
&\phi_{j+1}-d_1\phi_{j}=(d_2)^{j+M+1}(\phi_{-M}-d_1\phi_{-M-1}),\\
&\phi_{j+1}-d_2\phi_{j}=(d_1)^{j+M+1}(\phi_{-M}-d_2\phi_{-M-1}).
\end{align*}
Subtract the above equations, divide by $d_2-d_1$ and then with $d_1d_2=-1$, for $-M \leq j<-N$, we have
\vspace{-.1in}\begin{align}
\label{aux-eq_4}
\begin{split}
\phi_{j}&=\frac{(d_2)^{j+M+1}(\phi_{-M}-d_1\phi_{-M-1})-(d_1)^{j+M+1}(\phi_{-M}-d_2\phi_{-M-1})}{d_2-d_1}\\
&=\frac{(d_2)^{j+M+1}-(d_1)^{j+M+1}}{d_2-d_1}\phi_{-M}\\
&=\frac{(-1)^{j+M+1}(d_1)^{-(j+M+1)}-(d_1)^{j+M+1}}{d_2-d_1}\phi_{-M}.
\end{split}
\vspace{-.1in}\end{align}

Note that $d_1>1$, and both $(d_1)^{-x} +(d_1)^x$ and $-(d_1)^{-x} +(d_1)^x$ are increasing for $x \in \RR$. Thus the subsequences of $\phi_j$ with even $j=2k$ and odd $j=2k+1$ are increasing for $k=-M/2...-N/2-1$. Therefore it implies that $\ds\max_{-M \leq j<-N}\phi_{j}=\max\{\phi_{-N-2},\phi_{-N-1}\}$. Similarly, we have $\ds\max_{N < j\leq M}\phi_{j}=\max\{\phi_{N+1},\phi_{N+2}\}$. Then we must have $j_M \in  [-N-2,N+2]$.
There exists a subsequence $M_k$ of $M$ such that $\bar j=\ds\lim_{M_k \to \infty} j_{M_k}$ for some $\bar j \in [-N-2,N+2]$. Thus, $\phi_{\bar j}^{(\infty)}=1$. Moreover, by taking the limit, we have that $\Lambda \phi^{(\infty)}= \lambda_{\infty} \phi^{(\infty)}$. By the strong positivity of the semigroup generated by $\Lambda$, $\phi_{\bar j}^{(\infty)}=1>0$ implies that  $\phi_j^{(\infty)}>0$  for all $j$ and so $\lambda_{\infty}$ must be its principal eigenvalue. Therefore, $\ds\lambda=\lim_{M \to \infty}\lambda_{M}$.
\end{proof}

Indeed, Lemma \ref{lambda_limit-lemma} provides an alternative proof that $\lambda$ is a principal eigenvalue of \eqref{eigen-eq1}. Next we want to investigate the principal eigenvalue of \eqref{eigen-eq}. We have the following lemma.

\begin{lemma}
\label{p_solution-lemma}
   There exists a positive solution $\psi \in l^2$ satisfying $\|\psi\|_{\infty}=1$ to \eqref{eigen-eq1} for any $\gamma \geq \lambda$.
\end{lemma}
\begin{proof}
The proof follows by arguments in section 2.3 in \cite{Teschl}.
\end{proof}

Thus, by Lemma \ref{p_solution-lemma}, we can only expect a positive eigenvector to \eqref{eigen-eq} for $\gamma_\mu \geq \lambda$, where $\gamma_\mu $ is the associated eigenvalue of \eqref{eigen-eq}. It should be pointed out that we can not obtain the positive eigenvector directly by Lemma \ref{p_solution-lemma} except $\gamma_\mu=\lambda$, because the spectral theory of Jacobi operators in \cite{Teschl} is mainly on a Hilbert space $l^2$, while we require one for  \eqref{eigen-eq1} on $X_{\mu}$ or  for \eqref{eigen-eq} on $X$. Then we have the following Lemma.
\begin{lemma}
\label{p_solution-lem}
   There exists a positive solution $\phi \in X$ to \eqref{eigen-eq} for $\gamma_\mu = \lambda(\mu)$ for $\mu \in [\hat \mu, \mu^*)$, and $\phi_{k} =1$ for $k=n_0,n_0+1$ and $n_0>N$. Moreover, $\phi_j=1$ for $j>N$ and there is a positive number $l$ such that $\ds\lim_{j \to -\infty}\phi_j=\begin{cases} 0, \mu=\hat \mu,\\
l, \mu>\hat \mu,
\end{cases}$ if $\lambda>1$ and $\ds\lim_{j \to -\infty}\phi_j=l$ for any $\mu>0$ if $\lambda=1$.
\end{lemma}
\begin{proof}
We show how to obtain the principal eigenvectors to \eqref{eigen-eq} for $\lambda>1$ and the case $\lambda=1$ follows similarly as $\lambda>1$ with $\mu>\hat \mu$. 
For $|j|>N$, $a_j=1$. Thus, recalling \eqref{eigen-eq}, for $|j|>N$, we have 
\vspace{-.1in}\begin{equation*}
e^{-\mu}\phi_{j+1}-\phi_{j}+e^{\mu}\phi_{j-1}=\lambda(\mu) \phi_{j}.
\vspace{-.1in}\end{equation*}
Thus, recalling $\lambda(\mu)=e^{\mu}-1+e^{-\mu}$, for $|j|>N$, we have 
\vspace{-.1in}\begin{equation}
\label{eigen-eq_2}
e^{-\mu}\phi_{j+1}+e^{\mu}\phi_{j-1}=(e^{\mu}+e^{-\mu})\phi_{j}.
\vspace{-.1in}\end{equation}
Since $\phi_{k} =1$ for $k=n_0,n_0+1$, with \eqref{eigen-eq_2}, for $j>N$, $\phi_{j} =1$. On the other hand, for $j<-N$, we write \eqref{eigen-eq_2} as the following,
\vspace{-.1in}\begin{equation}
\label{eigen-eq_3}
\phi_{j-1}=(1+e^{-2\mu})\phi_{j}-e^{-2\mu}\phi_{j+1}.
\vspace{-.1in}\end{equation}
Let $c_1=1+e^{-2\mu}$ and $c_2=-e^{-2\mu}$. We can solve a recursive sequence $\phi_{j-1}=c_1\phi_{j}+c_2\phi_{j+1}$. To this end, we use an auxiliary equation $x^2-c_1x-c_2=0$. Then solve it to have two roots $d_1=1$ and $d_2=e^{-2\mu}$. Therefore, we have
\begin{align*}
&\phi_{j-1}-d_1\phi_{j}=d_2(\phi_{j}-d_1\phi_{j+1}),\\
&\phi_{j-1}-d_2\phi_{j}=d_1(\phi_{j}-d_2\phi_{j+1}).
\end{align*}
Thus, for $j<-N$, we have
\begin{align*}
&\phi_{j-1}-d_1\phi_{j}=(d_2)^{-N+1-j}(\phi_{-N}-d_1\phi_{-N+1}),\\
&\phi_{j-1}-d_2\phi_{j}=(d_1)^{-N+1-j}(\phi_{-N}-d_2\phi_{-N+1}).
\end{align*}
Subtract the above equations, divide by $d_2-d_1$ and then for $j<-N$, we have
\vspace{-.1in}\begin{align}
\label{aux-eq_4}
\begin{split}
\phi_{j}&=\frac{(d_2)^{-N+1-j}(\phi_{-N}-d_1\phi_{-N+1})-(d_1)^{-N+1-j}(\phi_{-N}-d_2\phi_{-N+1})}{d_2-d_1}\\
&=:C_1+C_2e^{2\mu j}.
\end{split}
\vspace{-.1in}\end{align}
Since $\ds \lim_{j \to -\infty}(d_2)^{-N+1-j}= \lim_{j \to -\infty}e^{-2\mu(-N+1-j)}=0$ and $(d_1)^{-N+1-j}=1$, 
$$
\lim_{j \to -\infty}\phi_{j}=\frac{-(\phi_{-N}-d_2\phi_{-N+1})}{d_2-d_1}.
$$
Thus, $\phi \in X$. Next, we prove that $\phi>0$. Suppose that $\phi_{k_0} \leq 0$ for some $k_0 < N$ while $\phi_{j} > 0$ for $j>k_0$ (i.e. $k_0$ is the first oscillation point around 0 from the right).
Let $\hat \phi$ be a solution with $\hat \phi_{k} =e^{2 \mu k}$ for $k=n_0,n_0+1$ and $n_0>N$. Then for $j>N$, $\hat \phi_{j} =e^{2 \mu j}$ for all $j>N$. There is an $\epsilon>0$ small enough such that $k_0$ is also an oscillation point for $\phi-\epsilon\hat \phi$. Then $\phi-\epsilon\hat \phi >0$ for $N < j < -\frac{ln(\epsilon)}{2 \mu}$ and $\phi-\epsilon\hat \phi<0$ for $j > -\frac{ln(\epsilon)}{2 \mu}$. Thus there exists another oscillation point for $\phi-\epsilon\hat \phi$. This causes a contradiction with \textbf{oscillation theory}, every solution can change sign at most once (See Corollary 2.8 in Section 2.3 of \cite{Teschl}), and so $\phi_j >0$. 

If $\mu=\hat\mu$, then $\lambda(\mu)=\lambda$. In this case, for $j<-N$, $e^{-\mu j}\phi_{j}=C_1 e^{-\mu j} + C_2 e^{\mu j}\in l^2 \subset X$, and it is a principal eigenvector of \eqref{eigen-eq1}. That implies that $C_1=0$ and so $\ds\lim_{j \to -\infty}\phi_j=0$.
Finally, we prove that $\ds\lim_{j \to -\infty}\phi_j>0$ for $\mu>\hat \mu$. $\phi>0$ implies that $\ds\lim_{j \to -\infty}\phi_j \geq 0$. Assuming that to be false with $\ds\lim_{j \to -\infty}\phi_j=0$, by \eqref{aux-eq_4}, $\phi_{j}=Ce^{2uj}$ for $j<N$. Thus
\vspace{-.1in}\begin{equation*}
\phi_{j}=\begin{cases} Ce^{2\mu j}, j<-N\\
\phi_{j}, j\in [-N,N]\\
1,j>N.
\end{cases}
\vspace{-.1in}\end{equation*}

Similarly, for $j>N$ with $\hat d_1=1$ and $\hat d_2=e^{2\mu}$, we have
\vspace{-.1in}\begin{align}
\label{aux-eq_5}
\begin{split}
\phi_{j}&=\frac{(\hat d_2)^{j-N+1}(\phi_{N}-\hat d_1\phi_{N-1})-(\hat d_1)^{-N+1-j}(\phi_{N}-\hat d_2\phi_{N-1})}{\hat d_2-\hat d_1}\\
&=:C_3+C_4e^{2\mu j}.
\end{split}
\vspace{-.1in}\end{align}

For $\mu>\hat \mu$, $\lambda(\mu)>\lambda(\hat \mu)=\lambda$. We can assume that there exists another independent positive solution $\tilde \phi$ (It is possible because there are at least two independent solutions $u_{\pm}(z,n)>0$ by Lemma 2.6 in Section 2.3 in \cite{Teschl}),
\vspace{-.1in}\begin{equation*}
\tilde\phi_{j}=\begin{cases} C_1+C_2e^{2\mu j}, j<-N\\
\tilde\phi_{j}, j\in [-N,N]\\
C_3+C_4e^{2\mu j},j>N.
\end{cases}
\vspace{-.1in}\end{equation*}
Then, by the independence and positivity of $\tilde \phi$ and $\phi$, we must have $C_1>0$ and $C_4>0$. Thus there exists at least one sign change at each end of the solution $\phi-\epsilon\tilde \phi$ for $j<-N$ or $j>N$ with small enough choice of $\epsilon$ . This causes a contradiction again with \textbf{oscillation theory}. Therefore, we must have $\ds\lim_{j \to -\infty}\phi_j=l>0$ for $\mu>\hat\mu$.
 \end{proof}

\subsection{Sub/Super-Solutions}
In this subsection, we construct a super-solution and a sub-solution with Lemma \ref{p_solution-lem}. By Lemma \ref{p_solution-lem}, the principal eigenvalue pair, denoted by $(\lambda_\mu, \phi^{\mu}_{j})$, exists for equation \eqref{eigen-eq}, where $\lambda_\mu=\lambda(\mu)$ for $\mu \in [\hat \mu, \mu^*)$.

Let
\vspace{-.05in}\begin{equation}
\label{v-super-eq1}
\bar u_{j}= e^{-\mu (j-ct)}\phi^{\mu}_{j}.
\vspace{-.05in}\end{equation}

\begin{lemma}
\label{super-solution-prop1}
   $\{\bar u_{j}\}_{j \in \ZZ}$ is a super-solution of
  \eqref{main-eq}.
\end{lemma}
\begin{proof}
 By (H1), we have $f_{j}(\bar u_{j})-f_j(0) \leq 0$. Recall that $a_j=f_j(0) $. By direct calculation, we have
\begin{align*}
&(\bar u_{j})_{t}- [\bar u_{j+1}-2\bar u_{j}+\bar u_{j-1}+f_{j}(\bar u_{j})\bar u_{j}]\\
\geq &(\bar u_{j})_{t}- [\bar u_{j+1}-2\bar u_{j}+\bar u_{j-1}+a_{j} \bar u_{j}]\\
=&0.
\end{align*}
\end{proof}

Let \vspace{-.05in}\begin{equation} 
\label{v-underbar-eq1}
\underline u_{j}= e^{-\mu (j-ct)}\phi^{\mu}_{j}-d_1 e^{-{\mu}_{1}( j-c t)}\phi^{\mu_{1}}_{j}.
\vspace{-.05in}\end{equation} for $\hat \mu \leq \mu< \mu_{1}<\min\{2 \mu,\mu^*\}$.

\begin{lemma}
\label{sub-solution-prop1}
   $\{\underline u_{j}\}_{j \in \ZZ}$ is a sub-solution of
  \eqref{main-eq} for any $d_1>\max\bigg\{\frac{\ds\sup_{j}\phi^{\mu}_{j}}{\ds\inf_{j}\phi^{\mu_{1}}_{j}},\frac{\ds L (\sup_{j}\phi^{\mu}_{j})^2}{(\mu_1 c-\lambda(\mu_{1}))\ds\inf_{j}\phi^{\mu_{1}}_{j}}\bigg\}.$
\end{lemma}
\begin{proof}
Let $f_{j}(\bar u_{j}) =f_{j}(0)$ if $\bar u_{j} \leq 0$. By Lemma \ref{Au-lm4}(3), for $\hat \mu \leq \mu< \mu_{1}<\min\{2 \mu,\mu^*\}$, 
\vspace{-.05in}\begin{equation}
\label{c-ineq1}
c(\mu_1)=\frac{\lambda(\mu_{1})}{\mu_{1}} \leq \frac{\lambda(\mu)}{\mu}=c.
\vspace{-.05in}\end{equation}
 Then if $\bar u_{j} \leq 0$, we have 
 \begin{align*}
&(\bar u_{j})_{t}- [\bar u_{j+1}-2\bar u_{j}+\bar u_{j-1}+f_{j}(\bar u_{j})\bar u_{j}]\\
&= (\bar u_{j})_{t}- [\bar u_{j+1}-2\bar u_{j}+\bar u_{j-1}+f_{j}(0) \bar u_{j}]\\
&=-({\mu}_{1}c-\lambda(\mu_{1}))d_1e^{-\mu_{1} (j-c t)}\phi^{\mu_{1}}_{j}\\
&\leq 0.
\end{align*}
By Lemma \ref{p_solution-lem} and $\hat \mu \leq \mu< \mu_{1}$, both $\ds\sup_{j}\phi^{\mu}_{j}$ and $\ds\inf_{j}\phi^{\mu_1}_{j}$ are positive. Let 
$$
d_0=\max\bigg\{\frac{\ds\sup_{j}\phi^{\mu}_{j}}{\ds\inf_{j}\phi^{\mu_{1}}_{j}},\frac{\ds L (\sup_{j}\phi^{\mu}_{j})^2}{(\mu_1 c-\lambda(\mu_{1}))\ds\inf_{j}\phi^{\mu_{1}}_{j}}\bigg\}.
$$
Note that $d_1>d_0$. If $\bar u_{j}>0$, we have $e^{-\mu (j - c t)}\phi^{\mu}_{j}-d_1 e^{-\mu_1 (j - c t)}\phi^{\mu_{1}}_{j}>0$ and then
$$e^{-(\mu-\mu_1)(j-ct)}>d_1 \frac{\phi^{\mu_{1}}_{j}}{\phi^{\mu}_{j}}\geq \frac{d_1}{d_0}\geq 1,$$
that implies that $j -ct \geq 0$. For $\bar u_{j}>0$, $\bar u_{j}^{2} \leq e^{-2 \mu(j-ct)}(\phi^{\mu}_{j})^{2}$. 
Then together with \eqref{c-ineq1}, for $\bar u_{j}>0$, we have 
\begin{align*}
&(\bar u_{j})_{t}- [\bar u_{j+1}-2\bar u_{j}+\bar u_{j-1}+f_{j}(\bar u_{j})\bar u_{j}]\\
&=(\bar u_{j})_{t}- [\bar u_{j+1}-2\bar u_{j}+\bar u_{j-1}+a_{j}\bar u_{j}]+f_{j}(0)\bar u_{j}-f_{j}(\bar u_{j})\bar u_{j}\\
&=(\bar u_{j})_{t}- [\bar u_{j+1}-2\bar u_{j}+\bar u_{j-1}+a_{j}\bar u_{j}]-f_{j}'(y)\bar u_{j}^{2}\\
&\leq -(\mu_1 c-\lambda(\mu_{1}))d_1e^{-\mu_{1} (j-ct)}\phi^{\mu_{1}}_{j}-f_{j}'(y)e^{-2 \mu(j-ct)}(\phi^{\mu}_{j})^{2}\\
&\leq e^{-\mu_{1} (j-ct)}[ -(\mu_1 c-\lambda(\mu_{1})) d_1\phi^{\mu_{1}}_{j}-f_{j}'(y)e^{-(2 \mu-\mu_{1})(j-ct)}(\phi^{\mu}_{j})^{2}]\\
&\leq e^{-\mu_{1} (j-ct)}[ -(\mu_1 c-\lambda(\mu_{1})) d_1\phi^{\mu_{1}}_{j}-f_{j}'(y)(\phi^{\mu}_{j})^{2}],
\end{align*}
where $y$ is such that $f_{j}(0)-f_{j}(\bar u_{j})=-f_{j}'(y)\bar u_{j}$. Recall that $\mu <\mu_{1}<\min\{2 \mu,\mu^*\}$ and in (H1), $\ds-L<\inf_{j\in\ZZ,u_j\geq 0}\{f_j'(u_j)\}\leq\sup_{j\in\ZZ,u_j\geq 0}\{f_j'(u_j)\}<0$ for some $L>0$. Since $d_1>d_0 \geq \frac{\ds L (\sup_{j}\phi^{\mu}_{j})^2}{(\mu_1 c-\lambda(\mu_{1}))\ds\inf_{j}\phi^{\mu_{1}}_{j}}$, we have $[ -(\mu_1 c-\lambda(\mu_{1})) d_1\phi^{\mu_{1}}_{j}-f_{j}'(y)(\phi^{\mu}_{j})^{2}] \leq 0$ and thus $$(\bar u_{j})_{t}- [\bar u_{j+1}-2\bar u_{j}+\bar u_{j-1}+f_{j}(\bar u_{j})\bar u_{j}]\leq 0$$ for $\bar u_{j}>0$. This completes the proof of the proposition.
\end{proof}
\begin{remark}
\label{remark-c31}
For $\mu=\mu^*$, Lemma \ref{super-solution-prop1} holds. However, Lemma \ref{sub-solution-prop1} does not hold for $\mu=\mu^*$, because valid positive eigenvector must locate in $(0, \mu^*]$ and there is no room for the choice of $\mu_1$. For another critical number $\mu=\hat \mu$, it is fine to be included, but should pay special attention to the difference of the tail property of the principal eigenvector.
\end{remark}
\subsection{Existence of Transition Fronts}
In the last subsection we constructed the super/sub-solutions on the interval $[\hat{\mu},\mu^*)$ of $\mu$. In this subsection, we can obtain the existence of transition fronts to \eqref{main-eq} for $c \in (c^*,\hat{c}]$ by the comparison principle. After that, with the limiting argument, we can have the existence of transition fronts to \eqref{main-eq} of $c =c^*$. The proof of existence of transition fronts in part (1) of main theorem is completed by the following proposition.

\begin{proposition}
Assume (H1)-(H2).
\label{existence-prop1}
If $\lambda \in [1,\lambda^{*})$ and $\hat{c}> c^*$, then transition fronts exist for any speed $c \in [c^{*}, \hat{c}]$.
Moreover, for $c^* < c < \hat c$, the constructed transition front $u_j(t)$ satisfy
\begin{equation}
\label{tail_estimate_limit}
\lim_{j - ct \to \infty}\frac{u_j(t)}{e^{-\mu (j-ct)}}=\phi^{\mu}_{j}.
\end{equation}
\end{proposition}

To prove Proposition \ref{existence-prop1}, we will apply the following lemma. Let $(\lambda_M,\phi_M)$ be principal eigenvalue and eigenvector pair of \eqref{truncated-eigen-eq} with $\|\phi_M\|_{\infty}=1$ and $\tilde u=\delta \phi_M$ for $\delta>0$. 
\begin{lemma}
\label{Truncated-sub}
For any given $M>>1$, there is a small enough $\delta_0>0$ such that $\tilde u$ is a sub-solution of \eqref{main-eq} for any $\delta \in (0,\delta_0)$.
\end{lemma}
\begin{proof}
Recall that $a_j=f_j(0) $. Choose $\delta_0$ small enough such that $$f_j(0)-f_j(\tilde u_{j}) \leq f_j(0)-f_j(\delta_0)<\lambda_M, \forall \delta \in (0,\delta_0).$$ By direct calculation, we have
\begin{align*}
&(\tilde u_{j})_{t}- [\tilde u_{j+1}-2\tilde u_{j}+\tilde u_{j-1}+f_{j}(\tilde u_{j})\tilde u_{j}]\\
= &(\tilde u_{j})_{t}- \tilde u_{j+1}-2\tilde u_{j}+\tilde u_{j-1}+a_{j} \tilde u_{j}]+ (f_j(0)-f_j(\tilde u_{j}))\tilde u_{j}\\
=&(-\lambda_M +(f_j(0)-f_j(\tilde u_{j})))\tilde u_{j}\\
\leq& 0.
\end{align*}
\end{proof}

\begin{proof}[Proof of Proposition \ref{existence-prop1}]
As long as we have the required super/sub-solutions, the existence of transition fronts can be obtained by the standard $''$squeeze$''$ techniques. Indeed, if $\lambda \in [1,\lambda^{*})$ and $\hat{c}> c^*$, then we have a positive principal eigenvector to \eqref{eigen-eq} for any speed $c \in (c^{*}, \hat{c}]$. Let $\bar u$ and $\underline u$ be chosen as in \eqref{v-super-eq1} and \eqref{v-underbar-eq1}, $v=\min\{\bar u, u^*\}$ and $w=\max\{\underline u,0\}$. Following arguments similar to \cite{GuHa}, we have an entire solution that is sandwiched between $v$ and $w$. In fact, for each $n \in \NN$, let $\{u_j^n\}_{j \in \ZZ}$ be a solution of \eqref{main-eq} with initial condition $u_j^n(-n)=v_{j}(-n)$. With the comparison principle, we have that for any $n \in \NN$, and $(t, j) \in (-n, \infty) \times \RR$,
$$0 \leq w_j(t) \leq u^n_j(t) \leq v_j(t) \leq u_j^*.$$
In particular, letting $t=-n+1$, we have $u^n_j(-n+1) \leq v_j(-n+1)=u_j^{n-1}(-n+1),$ for all $n \in \NN$ and $j \in \ZZ$. With the comparison principle again, we have that for any $n \in \NN$, and $(t, j) \in (-n+1, \infty) \times \RR$,
$$0 \leq u^n_j(t) \leq u^{n-1}_j(t)\leq u_j^*.$$ Note that $|u^{n}_j(t)|\leq u_j^*$ and $|\dot u^{ n}_j(t))|\leq C\|\Lambda\|+ \ds\max_{0 \leq v\leq u_j^*} |f_j(v)| \max_{j}u_j^*$ because $\Lambda$ is a bounded operator with operator norm $\|\Lambda\|$. By Arzel$\grave{a}-$ Ascoli theorem, there exists a subsequence $\{u^{{n}_k}_j(t)\}_{j \in \ZZ}$ with ${n}_k>|t|+1$, such that it converges uniformly on bounded sets. Letting $n_k \to \infty$, $u_{j}(t):=\ds\lim_{n_k \to \infty}u^{{n}_k}_j(t)$ for all $(t,j) \in \RR \times \ZZ$. Integrating \eqref{main-eq} over $[0,t]$ with each $u^{n}_j(t)$ for $n \in \NN$, we have 
$$
{u}^{n}_{j}(t)={u}^{n}_{j}(0)+\int_0^{t}[u^n_{j+1}-2u^n_{j}+u^n_{j-1}+f_j(u^n_{j})u^n_{j}]d s.
$$
Letting $n \to \infty$, we have 
$$
{u}_{j}(t)={u}_{j}(0)+\int_0^{t}[u_{j+1}-2u_{j}+u_{j-1}+f_j(u_{j})u_{j}]ds,
$$
which implies that ${u}_{j} \in C^1$ and also satisfies \eqref{main-eq}. Moreover, we also have that $$0 \leq w_j(t) \leq u_j(t) \leq v_j(t) \leq u_j^*.$$ Thus, it yields $\ds\lim_{j \to \infty}u_j(t) =0$. It remains to show that $\ds\lim_{j \to -\infty}u_j(t) =u_j^*$. By strong comparison principle, we have $u_j(\tau)>0$ for $\tau>0$. Let $\tilde u$ be as in Lemma \ref{Truncated-sub}. Since $\tilde u$ is compactly supported on $[-M,M]$, there exists a $\delta \in (0,\delta_0)$, such that $u_j(\tau)> \tilde u$. With the comparison principle again, $u_j(t) \geq u_j(t-\tau;\tilde u)$ for $t>\tau$, where $\{u_j(t-\tau;\tilde u)\}_{j \in \ZZ}$ is the solution of \eqref{main-eq} with initial $\tilde u$ at $t=\tau$. Due to uniqueness of positive stationary solution of \eqref{main-eq}, we must have $\ds\lim_{t \to \infty } u_j(t-\tau;\tilde u)=u_j^*$. Then for all $j \in \ZZ$, $$\ds\liminf_{t \to \infty} u_j(t) \geq \lim_{t \to \infty } u_j(t-\tau;\tilde u)=u_j^*,$$ that implies that $\ds\lim_{t \to \infty } u_j(t)=u_j^*$. By the definition of $w(t)$ (sub-solution), there exist positive large $L$ and small $\sigma$ such that, for $j-ct>L$, 
$$
w(t) \geq \sigma e^{-\mu(j-ct)}> \sigma e^{-\mu L} \geq \tilde u.
$$
In particular, let $\tilde t=\frac{j-L}{c}$ and we have $u_j(\tilde t)\geq w(\tilde t) \geq \sigma e^{-\mu L} \geq \tilde u$. Since $\ds\lim_{t \to \infty } u_j(t;\tilde u)=u_j^*$, for any $\epsilon>0$, there exists a $T_0>0$ such that
$u_j(t;\tilde u)>u_j^*-\epsilon$, for all $t>T_0$ and $j \in\ZZ$.
Note that as $j \to -\infty$, $t-\tilde t \to \infty$ for given $t \in \RR$. Then for $t-\tilde t>T_0$,
$$u_j(t)=u_j(t-\tilde t+ \tilde t)\geq u_j(t-\tilde t;\tilde u)>u_j^*-\epsilon,$$  thus implies that $\ds\lim_{j \to -\infty } u_j(t)=u_j^*$.

For $c=c^*$ ($\mu=\mu^*$), we claim that the transition front also exists and shall prove it by limit arguments due to the invalid sub-solutions in Remark \ref{remark-c31}. To prove the case with $c =c^{*}$, pick a sequence $\hat{c}>c_{\bar n}>c^*$ such that $c_{\bar n} \to c^*$. We simply denote the transition fronts of speed $c_{\bar n}$ by $\{u^{\bar n}_j(t)\}_{j \in \ZZ}$. 
By similar limiting arguments above for $\{u^{n}_j(t)\}_{j \in \ZZ}$, let the transition front of speed $c^*$ be $u^{\dag}_{j}(t):=\ds\lim_{\bar n_k \to \infty}u^{{\bar n}_k}_j(t)$. 

Finally, for $c^* < c < \hat c$, the limit \eqref{tail_estimate_limit} follows from $w_j(t) \leq u_j(t) \leq v_j(t)$ for all $j$ and $t>0$ with the comparison principle. This completes the proof.
\end{proof}

By Proposition \ref{existence-prop1}, we have the following exponential tail estimates for the constructed transition fronts.
\begin{corollary}
For the constructed transition fronts of $c^* < c < \hat c$ in Proposition \ref{existence-prop1}, they own exponential tail estimates: for any $\epsilon>0$, there exist $C_1,\ C_2,\ T>0$ such that for $t>T$ and $j>ct$,
\begin{equation}
\label{tail_estimate_t}
C_1e^{-(\mu+\epsilon)(j-ct)} \leq u_j(t) \leq C_2e^{-(\mu-\epsilon)(j-ct)}.
\end{equation}
\end{corollary}

\begin{remark}
If $\lambda=1$, $\hat{c}=\infty$ ($\hat{\mu}=0$). This includes the case of homogeneous equation with $f_j(u_j)=1-u_j$. In these cases, the required positive eigenvectors to \eqref{eigen-eq} are always available for any $\mu \in (0,\mu^*)$. For $c=c^*$, since comparison principle does not work due to invalid sub-solutions, the tail estimate remains an open question and we should pay special attention to the critical speed $c^*$.
\end{remark}

\subsection{Tail Estimates of Transition Fronts}
In the last subsection, for any constructed transition fronts, they satisfy an exponential tail estimate \eqref{tail_estimate_t}. In this subsection, we will prove that if transition fronts exist, then they must own similar exponential tail estimates, that completes the proof of part (1) of main theorem. Recall that $\lambda(\mu)=e^{\mu}-1+e^{-\mu}$ and $c(\mu)=\frac{\lambda(\mu)}{\mu}$ for $\mu>0$, then we have the following propositions about the tail estimates of transition fronts.
\begin{proposition}
\label{Main-lm-up}
Let $c>c^{*}$, and $u_{j}(t)$ be a transition front of  \eqref{main-eq} with speed c. Then for any $\epsilon>0$, there exists a $\hat{K}_{\epsilon}>0$ such that $$ u_{k}(t_{j}) \leq \hat{K}_{\epsilon} e^{-(\mu-\epsilon)(k-j)},$$ for $k\geq j$ with $j,t_j$ as in Definition  \ref{meanspeed-def} of Mean Wave Speed.
\end{proposition}
\begin{proof}
Suppose not, then there exist $\epsilon$, $j_n, t_{j_n}$, $k_{n}$ and $x_n:=k_n-j_n \to \infty$ such that
\vspace{-.1in}\begin{equation}
\label{eq_4_3}
 u_{k_{n}}(t_{j_n}) \geq \hat{K}_{\epsilon} e^{-(\mu-\epsilon)x_n}.
 \vspace{-.1in}\end{equation}
For simplicity, we denote $T=t_{j_n}$.

Recall \eqref{nonlinear-sol} that for $T \geq 0$,
\begin{align*}
u_j(t)&=e^{t-T}\sum_kh^{\ZZ}_{2(t-T)}(j-k)u_k(T)-\int_{T}^{t}e^{(t-s)}\sum_kh^{\ZZ}_{2(t-s)}(j-k)g_k(s)ds\\
&:=A(t)-B(t),
\end{align*}
and $g_j(t)=(1-f_j(u_j))u_j$.

Let $z_0=csch(\mu)$. Recall that $c>\frac{2}{z_0}$ by Lemma \ref{Au-lm4} (4). Let $\tilde{t}$ be such that $x_n=(c-\frac{2}{z_0})\tilde{t}$ and $j=k_n+\frac{2}{z_0} \tilde{t}=j_n+c \tilde{t}>N$. Choose $t=\tilde{t}+T$, that is, $\tilde{t}=t-T$. Then as $x_n \to \infty$, $\tilde t \to \infty$ and thus $t \to \infty$. By heat kernel estimate \eqref{heatkernel} and Lemma \ref{Au-lm4}, with \eqref{eq_4_3} we have
\begin{align*}
A(t)&= e^{\tilde{t}}\sum_{k} h_{2 \tilde{t}}^{\ZZ}(j-k) u_{k}(T)\\
&\geq C e^{\tilde{t}}\sum_{k} F(2\tilde{t},j-k) u_{k}(T)\\
&= C \frac{1}{\sqrt{2\pi}}(\frac{e^{ \tilde{t}}}{(1+4 \tilde{t}^2)^{\frac{1}{4}}} u_{j}(T)+ \sum_{k\neq j}\frac{exp[-\tilde{t}+ |j-k|\varsigma(2\tilde{t}/|j-k|)]}{(1+4 \tilde{t}^2+|j-k|^2)^{\frac{1}{4}}} u_{k}(T))\\
&\geq C \frac{1}{\sqrt{2\pi}}\frac{exp[-\tilde{t}+|j-k_n|\varsigma(2\tilde{t}/|j-k_n|)]}{(1+4(\tilde{t})^2+(j-k_n)^2)^{\frac{1}{4}}} u_{k_n}(T)\\
&\geq C\hat{K}_{\epsilon}  \frac{1}{\sqrt{2\pi}}\frac{exp[-\tilde{t}+|j-k_n|\varsigma(2\tilde{t}/|j-k_n|)]}{(1+4(\tilde{t})^2+(j-k_n)^2)^{\frac{1}{4}}} e^{-(\mu-\epsilon)x_n}\\
&=C\hat{K}_{\epsilon}  \frac{1}{\sqrt{2\pi}}\frac{exp[-\tilde{t}+\frac{2\tilde{t}}{z_0}\varsigma(z_0)-(\mu-\epsilon)(c-\frac{2}{z_0})\tilde{t}]}{(1+4(\tilde{t})^2+(j-k_n)^2)^{\frac{1}{4}}} \\
&=C \hat{K}_{\epsilon} \frac{1}{\sqrt{2\pi}}\frac{exp[-\tilde{t}+\frac{2\tilde{t}}{z_0}\varsigma(z_0)-\mu(c-\frac{2}{z_0})\tilde{t}]}{(1+4(\tilde{t})^2+(j-k_n)^2)^{\frac{1}{4}}} e^{\epsilon (c-\frac{2}{z_0})\tilde{t}} \\
&=C \hat{K}_{\epsilon} \frac{1}{\sqrt{2\pi}}\frac{exp[\epsilon(c-\frac{2}{z_0})\tilde{t}]}{(1+4(\tilde{t})^2+(j-k_n)^2)^{\frac{1}{4}}} \\
&\geq e^{\tilde{\epsilon} \tilde{t}},
\end{align*}
where $\tilde{\epsilon}$ is chosen such that $\epsilon(c-\frac{2}{z_0})>\tilde{\epsilon}>0$ for $\mu \in (0, \mu^*)$ and the above inequality holds as $\tilde{t}$ is chosen large enough. On the other hand, we have that
\begin{align*}
B(t)&= \int_{T}^{t} e^{t-s}\sum_{k} h_{2 (t-s)}^{\ZZ}(j-k) g_{k}(s) ds\\
&= \int_{0}^{\tilde t} e^{\tilde t-s}\sum_{k} h_{2 (\tilde t-s)}^{\ZZ}(j-k) g_{k}(s) ds\\
&\leq C \int_{0}^{\tilde t} e^{\tilde t-s}\sum_{k} F(2(\tilde t-s),j-k) g_{k}(s) ds\\
&= C \int_{0}^{\tilde t} e^{\tilde t-s}\bigg(\sum_{k \leq -c (\tilde t -s)+c s+j_n}+\sum_{-c (\tilde t -s)+cs+j_n < k \leq c s+j_n}+\sum_{k>  c s+j_n}\bigg) F(2(\tilde t-s),j-k) g_{k}(s) ds\\
&:=B_{1}+B_{2}+B_{3}.
\end{align*}
Since $u_j$ is bounded, there exists a positive M such that
\vspace{-.1in}\begin{equation}
\label{g_ubound}
|g_j(s)|=|f_j(u_j)u_j-u_j| < M.
 \vspace{-.1in}\end{equation}
 For $k > cs+j_n>N$, $f_k(0)=a_k=1$. For any given positive $\eta$, there exists an $n_0$ such that for $n>n_0$, we have $u_k(s) < \eta$ whenever $k > cs+j_n$. This can be done because $\ds\lim_{k-cs \to \infty} u_k(s) =0$ and $j_n \approx ct_{j_n}$ large. Then for any $\tilde{\epsilon}>\delta>0$, there exists an $\eta$ such that $f_k(0)-f_k(u_k)< \delta$ and then 
\vspace{-.1in}\begin{equation}
\label{gk_ubound}
g_{k}(s)=(1-f_k(u_k)) u_k(s)< \delta u_k(s).
 \vspace{-.1in}\end{equation}
 Without loss of generality, let $j_n=0$ by translation and so $j=c \tilde t$. Recall that, in Lemma \ref{Au-lm4} (1), numerical computation shows that $l_0>0.66$. On the other hand, we have $c \geq c^*=\ds\inf_{\mu>0}\frac{e^{\mu}-1+e^{-\mu}}{\mu}\approx  2.073$, and thus $\frac{1}{c}<0.5<l_0$. Therefore by Lemma \ref{Au-lm4} (1), $\varsigma(\frac{1}{c})< \varsigma(l_0)$, that is, $\varsigma(\frac{1}{c})<0$. For $k \leq -c (\tilde t -s)+cs$, 
\vspace{-.1in}\begin{equation}
\label{fv_ubound}
\varsigma(\frac{2(\tilde t -s)}{j-k}) \leq \varsigma(\frac{1}{c})<0.
 \vspace{-.1in}\end{equation}

Therefore, with \eqref{g_ubound} and \eqref{fv_ubound}, for $B_{1}$ we have 
\begin{align*}
 B_{1}&=  C \int_{0}^{\tilde t} e^{\tilde t-s}(\sum_{k \leq -c (\tilde t -s)+cs} F(2(\tilde t-s),j-k) g_{k}(s) )ds\\
 &\leq C  M \int_{0}^{\tilde t} e^{\tilde t-s}\big (\sum_{k \leq -c (\tilde t -s)+cs} F(2(\tilde t-s),j-k) \big )ds\\
 &=C_1  \int_{0}^{\tilde t} \sum_{k \leq -c (\tilde t -s)+cs}\frac{e^{-(\tilde t -s)+(j-k)\varsigma(\frac{2(\tilde t -s)}{j-k})}}{\sqrt{2\pi}(1+4(\tilde t -s)^2+(j-k)^2)^{\frac{1}{4}}} ds\\
  &\leq C_1  \int_{0}^{\tilde t} \sum_{k \leq -c (\tilde t -s)+cs}\frac{e^{-(\tilde t -s)+(j-k)\varsigma(\frac{1}{c})}}{\sqrt{2\pi}(1+4(\tilde t -s)^2+(j-k)^2)^{\frac{1}{4}}} ds\\
  &\leq C_1  \int_{0}^{\tilde t} \sum_{k \leq min\{-c (\tilde t -s)+cs,0\}}\frac{e^{-(\tilde t -s)+(j-k)\varsigma(\frac{1}{c})}}{\sqrt{2\pi}(1+4(\tilde t -s)^2+(j-k)^2)^{\frac{1}{4}}} ds\\
  & \quad\quad + C_1  \int_{0}^{\tilde t} \sum_{0 < k \leq |-c (\tilde t -s)+cs|}\frac{e^{(-1+2c \varsigma(\frac{1}{c}))(\tilde t -s)}}{\sqrt{2\pi}(1+4(\tilde t -s)^2+(j-k)^2)^{\frac{1}{4}}} ds\\
    &\leq C_1  \int_{0}^{\tilde t} \sum_{k \leq 0}\frac{e^{-(\tilde t -s)+(j-k)\varsigma(\frac{1}{c})}}{\sqrt{2\pi}(1+4(\tilde t -s)^2+(j-k)^2)^{\frac{1}{4}}} ds+ C_1  \int_{0}^{\tilde t} \sum_{0 < k \leq c (\tilde t +2s)} 1 ds\\
&\leq C_1   \int_{0}^{\tilde t}\big ( \sum_{k \leq 0}e^{(-k)\varsigma(\frac{1}{c})} +  c (\tilde t +2s) \big)ds, \text{by } -(\tilde t -s)+j \varsigma(\frac{1}{c})<0 \\
&= C_1 \int_{0}^{\tilde t}\big ( \frac{e^{\varsigma(\frac{1}{c})}}{1-e^{\varsigma(\frac{1}{c})}}+  c (\tilde t +2s) \big)ds\\
& \leq P_{1}(\tilde t),
\end{align*}
where $C_1=CM$ and $\ds P_{1}(\tilde t)= C_1(2 c \tilde t^2+(\frac{e^{\varsigma(\frac{1}{c})}}{1-e^{\varsigma(\frac{1}{c})}} )\tilde t)$ that is a quadratic equation.

Let $\ds-\sigma=-1+2 \frac{\varsigma(z_1)}{z_1}=\max_{-c (\tilde t -s)+cs<k \leq cs} -1+2 \frac{\varsigma(z)}{z}$ with $z=\frac{2 (\tilde t -s)}{j-k}$ and $z_1=\frac{2}{c}$. We remark that  $c \geq c^* \approx 2.073$ and $z_1  \leq \frac{2}{c^*} \approx 0.9648$. Thus, $$\ds\sigma=1-2 \frac{\varsigma(z_1)}{z_1} \geq 1-2 \frac{\varsigma(0.9648)}{0.9648} \approx 0.0355793>0.$$ Then, for $B_{2}$ we have that
\begin{align*}
B_{2}&=C \int_{0}^{\tilde t} e^{\tilde t-s}(\sum_{-c (\tilde t -s)+cs < k \leq cs} F(2(\tilde t-s),j-k) g_{k}(s) )ds\\
 &\leq C  M \int_{0}^{\tilde t} e^{\tilde t-s}\big (\sum_{-c (\tilde t -s)+cs < k \leq cs} F(2(\tilde t-s),j-k) \big )ds\\
 &=C_1  \int_{0}^{\tilde t} \sum_{-c (\tilde t -s)+cs < K \leq cs}\frac{e^{-(\tilde t -s)+(j-k)\varsigma(\frac{2(\tilde t -s)}{j-k})}}{\sqrt{2\pi}(1+4(\tilde t -s)^2+(j-k)^2)^{\frac{1}{4}}} ds\\
 &\leq C_1  \int_{0}^{\tilde t} \sum_{-c (\tilde t -s)+cs < k \leq cs}\frac{e^{(-1+c \varsigma(\frac{2}{c}))(\tilde t -s)}}{\sqrt{2\pi}(1+4(\tilde t -s)^2+(j-k)^2)^{\frac{1}{4}}} ds\\
 &\leq C_1  \int_{0}^{\tilde t} \sum_{-c (\tilde t -s)+cs < k \leq cs}\frac{e^{-\sigma(\tilde t -s)}}{\sqrt{2\pi}(1+4(\tilde t -s)^2+(j-k)^2)^{\frac{1}{4}}} ds\\
  &\leq C_1  \int_{0}^{\tilde t} \sum_{-c (\tilde t -s)+cs < k \leq cs} 1 ds\\
  &= C_1  \int_{0}^{\tilde t} c (\tilde t -s)  ds\\
& \leq P_{2}(\tilde t),
\end{align*}
where $C_1=CM$ and $P_{2}(\tilde t)= C_1 (\frac{c}{2}\tilde t^2)$ that is a quadratic equation.

 Finally, with \eqref{gk_ubound}, for $B_{3}$ we have that 
\begin{align*}
B_{3}&=C \int_{0}^{\tilde t} e^{\tilde t-s}(\sum_{ k > cs} F(2(\tilde t-s),j-k) g_{k}(s) )ds\\
 &\leq \delta C \int_{0}^{\tilde t} e^{\tilde t-s}(\sum_{ k > cs} F(2(\tilde t-s),j-k) u_{k}(s) )ds\\
 &\leq \delta C \int_{0}^{\tilde t} e^{\tilde t-s}(\sum_{k} F(2(\tilde t-s),j-k) u_{k}(s) )ds.
\end{align*}

Note that 
\begin{align*}
A- B_{3}&\geq (S(t-T)u(T))_j-\delta(\int_{T}^{t}S(t-s)u(s)ds)_j\\
&=e^{-\delta \tilde t}(S(t-T)u(T))_j \\
&=e^{-\delta \tilde t}A(t) \\
&\geq e^{(\tilde \epsilon-\delta) \tilde t},
\end{align*} which is an exponential equation. On the other hand, $B_{1}+B_{2} \leq P_{1}(\tilde t)+P_{2}(\tilde t)$, which is a quadratic equation.
Thus, $u_j(t) \to \infty$ as $t \to \infty$ which contradicts that $u_j(t)$ is bounded.
\end{proof}

\begin{proposition}
\label{main-lm4}
Let $u_{j}(t)$ be a transition front of \eqref{main-eq} with speed c larger than $c^*$. Then for any $\epsilon>0$, there exists a $\tilde{K}_{\epsilon}>0$ and $T>0$ such that  $$ u_{k}(t_{j}) \geq \tilde{K}_{\epsilon} e^{-(\mu+\epsilon)(k-j)},$$ for $t_{j}>T$ and $k\geq j$ with $j,t_j$ as in Definition  \ref{meanspeed-def} of Mean Wave Speed.
\end{proposition}
\begin{proof}
We prove this lemma by contradiction. Assume the proposition to be false. Then for given $\epsilon$, there exist sequences $t_{j_n} \in \RR^+$, $k_n \in \ZZ^+$ and $j_n \in \ZZ^+$ such that $k_n \geq j_n$ and
\vspace{-.1in}\begin{equation}
\label{u-estimate}
u_{k_n}(t_{j_n}) \leq \tilde{K}_{\epsilon} e^{-(\mu+\epsilon)(k_n-j_n)}.
\vspace{-.1in}\end{equation}
 By applying Harnack inequality and shifting the origin of time and space, we can have a $q>0$ such that
\vspace{-.1in}\begin{equation}
\label{Harnack-assumption}
u_{k}(t_{j_n})\leq C e^{-(\mu+\frac{\epsilon}{2})(k_n-j_n)}, \forall k \in [(1-q\epsilon)k_n,(1+q\epsilon)k_n].
\vspace{-.1in}\end{equation}

Let $j= c t_{j_n}>N$, where $t_{j_n} \in \RR^+$ is chosen such that $j \in \ZZ^+$ and $N$ is as in (H2). For simplicity, we let $t=t_{j_n}$. We remark that $t$ is a sequence and $n \to \infty$ implies that $t \to \infty$. Recall \eqref{nonlinear-sol} that for $T = 0$,
\begin{align*}
u_j(t)&=e^{t}\sum_kh^{\ZZ}_{2t}(j-k)u_k(0)-\int_{0}^{t}e^{(t-s)}\sum_kh^{\ZZ}_{2(t-s)}(j-k)g_k(s)ds\\
&:=A(t)-B(t),
\end{align*}
where $g_j(t)=(1-f_j(u_j))u_j$, and $\begin{cases}
A(t)=e^{t}\ds\sum_kh^{\ZZ}_{2t}(j-k)u_k(0)\\
B(t)=\int_{0}^{t}e^{(t-s)}\ds\sum_k h^{\ZZ}_{2(t-s)}(j-k)g_k(s)ds.
\end{cases}$. 

We claim that $u_j(t)=A(t)-B(t)<0$ as $t \to \infty$, which causes a contradiction.

For any $\delta>0$, there exists a $l,j_{\delta}>0$, such that $u_k(s)\geq l$ for $N<k\leq (c-\delta)s-j_{\delta}$ and $s \geq 0$. Then for $N< k \leq (c-\delta)s-j_{\delta}$ and $s \geq 0$,
\vspace{-.1in}\begin{equation}
\label{g-lbound}
g_k(s)=(f_k(0)-f_k(u_k))u_k(s) \geq (1-\sup_{k}f_k(l))l:=\hat l.
\vspace{-.1in}\end{equation}

Thus, letting $\hat k = (c-\delta)s-j_{\delta}$ and $0< \sigma_2<\sigma_1<<1$, let 
\vspace{-.1in}\begin{equation*}
C(s,t)=-(t-s)+ |ct-\hat k|\varsigma(2(t-s)/|ct-\hat k|).
\vspace{-.1in}\end{equation*}
For $0< (1-\sigma_1)t \leq s \leq (1-\sigma_2)t$, choosing $\delta=\sigma_1$, we have
\vspace{-.1in}\begin{align*}
C(s,t)&=-(t-s)+ |ct-\hat k|\varsigma(2(t-s)/|ct-\hat k|)\\
&=-(t-s)+ (c(t-s)+\delta s+j_{\delta}) \varsigma(2(t-s)/(c(t-s)+\delta s+j_{\delta}))\\
& \geq - \sigma_1 t+  (c(t-s)+\delta s+j_{\delta})  \varsigma(\frac{2}{c+\delta (\frac{s}{t-s})+\frac{j_{\delta}}{t-s}})\\
& \geq - \sigma_1 t+  (c \sigma_2 t+\delta  (1-\sigma_1) t+j_{\delta})  \varsigma(\frac{2}{c+\delta (\frac{1-\sigma_2}{\sigma_1})+\frac{j_{\delta}}{\sigma_2 t}})\\
& = - \sigma_1 t+  (c \sigma_2 t+\sigma_1 (1-\sigma_1) t+j_{\delta})  \varsigma(\frac{2}{c+1-\sigma_2+\frac{j_{\delta}}{\sigma_2 t}})\\
& =( - \sigma_1+  (c \sigma_2 +\sigma_1 (1-\sigma_1))  \varsigma(\frac{2}{c+1-\sigma_2+\frac{j_{\delta}}{\sigma_2 t}})) t+j_{\delta}  \varsigma(\frac{2}{c+1-\sigma_2+\frac{j_{\delta}}{\sigma_2 t}}).
\vspace{-.1in}\end{align*}

For $t>\frac{j_{\delta}}{\sigma_2 ^2}$, we have
\vspace{-.1in}\begin{align}
\label{C-lbound}
\begin{split}
C(s,t)& \geq ( - \sigma_1+  (c \sigma_2 +\sigma_1 (1-\sigma_1))  \varsigma(\frac{2}{c+1-\sigma_2+\frac{j_{\delta}}{\sigma_2 t}})) t+j_{\delta}  \varsigma(\frac{2}{c+1-\sigma_2+\frac{j_{\delta}}{\sigma_2 t}})\\
& \geq ( - \sigma_1+  (c \sigma_2 +\sigma_1 (1-\sigma_1))  \varsigma(\frac{2}{c+1})) t+j_{\delta}  \varsigma(\frac{2}{c+1})\\
&:=-  \hat\sigma_1 t- \hat\sigma_2.
\end{split}
\vspace{-.1in}\end{align}
We remark that as $ \sigma_1 \to 0$, $\hat\sigma_1 \to 0$, that is, $\hat\sigma_1$ can be chosen as small as required by choosing small enough $\sigma_{1,2}$. 
Let 
\begin{align*}
\tilde C_B:&=\ds\lim_{t \to \infty}(C \hat{l} \frac{(\sigma_1-\sigma_2)\sqrt{t}}{(1+4 \sigma_1^2t^2+((\delta -\sigma_1(c-\delta))t+j_{\delta})^2)^{\frac{1}{4}}})\\
&=C \hat{l} \frac{(\sigma_1-\sigma_2)}{(1+4 \sigma_1^2+((\delta -\sigma_1(c-\delta)))^2)^{\frac{1}{4}}}
\end{align*} and $C_B=\tilde C_B/2$. Then, there exists a $T_B$ such that for $t>T_B$, 
\begin{equation*}
C \hat{l} \frac{(\sigma_1-\sigma_2)\sqrt{t}}{(1+4 \sigma_1^2t^2+((\delta -\sigma_1(c-\delta))t+j_{\delta})^2)^{\frac{1}{4}}} \geq C_B.
\end{equation*}
Therefore, with \eqref{heatkernel}, \eqref{g-lbound} and \eqref{C-lbound}, for $t>\max\{\frac{j_{\delta}}{\sigma_2 ^2},T_B\}$, we have
\begin{align*}
B(t)&=\int_0^{t}e^{t-s}\sum_{k} h_{2(t-s)}^{\ZZ}(j-k) g_{k}(s)ds\\
&\geq C \int_0^{t}e^{t-s}\sum_{k} F(2(t-s), j-k)g_{k}(s)ds\\
&\geq C \int_0^{t}e^{t-s}F(2(t-s),j-\hat k)g_{\hat k}(s)ds\\
&\geq C \hat{l} \int_0^{t}\frac{exp[-(t-s)+ |j-\hat k|\varsigma(2(t-s)/|j-\hat k|)]}{(1+4(t-s)^2+|j-\hat k|^2)^{\frac{1}{4}}} ds\\
&\geq C \hat{l} \int_{(1-\sigma_1)t}^{(1-\sigma_2)t}\frac{exp[-(t-s)+ |j-\hat k|\varsigma(2(t-s)/|j-\hat k|)]}{(1+4(t-s)^2+|j-\hat k|^2)^{\frac{1}{4}}} ds\\
&\geq C \hat{l}  e^{-\hat \sigma_1 t-\hat\sigma_2} \int_{(1-\sigma_1)t}^{(1-\sigma_2)t}\frac{1}{(1+4(t-s)^2+|j-\hat k|^2)^{\frac{1}{4}}} ds\\
&\geq C \hat{l}  e^{-\hat \sigma_1 t-\hat\sigma_2} \int_{(1-\sigma_1)t}^{(1-\sigma_2)t}\frac{1}{(1+4 \sigma_1^2t^2+((\delta -\sigma_1(c-\delta))t+j_{\delta})^2)^{\frac{1}{4}}} ds\\
&= C \hat{l}  e^{-\hat \sigma_1 t-\hat\sigma_2} \frac{(\sigma_1-\sigma_2)t}{(1+4 \sigma_1^2t^2+((\delta -\sigma_1(c-\delta))t+j_{\delta})^2)^{\frac{1}{4}}}\\
&= (C \hat{l} \frac{(\sigma_1-\sigma_2)\sqrt{t}}{(1+4 \sigma_1^2t^2+((\delta -\sigma_1(c-\delta))t+j_{\delta})^2)^{\frac{1}{4}}})\sqrt{t} e^{-\hat \sigma_1 t-\hat\sigma_2}\\
&\geq C_B \sqrt{t} e^{-\hat \sigma_1 t-\hat\sigma_2}.
\end{align*}
On the other hand, we have that
\begin{align*}
A(t)&= e^{t}\sum_{k} h_{2t}^{\ZZ}(j-k) u_{k}(0)\\
&\leq C_1 e^{t}\sum_{k} F(2t,j-k) u_{k}(0)\\
&= C_1 e^{t}[(\sum_{k \leq -ct} +\sum_{-ct<k \leq 0} +\sum_{1 \leq k  \leq (1-q\epsilon) k_n} + \sum_{(1-q\epsilon) k_n < k \leq (1+ q\epsilon) k_n}\\
&\quad \quad \quad \quad +\sum_{(1+ q\epsilon) k_n <  k \leq j-1} + \sum_{k = j} +\sum_{k >j}) F(2t,j-k) u_{k}(0)]\\
&:=A_1+A_2+A_3+A_4+A_5+A_6+A_7.
\end{align*}

For $k\leq -c t$, we have that $\varsigma(\frac{2t}{j-k}) \leq \varsigma(\frac{1}{c})<0$. Then for $A_1$ and t large, we have
\begin{align*}
 A_{1}:&= C_1 e^{t}\sum_{k \leq -c t} F(2t,j-k) u_{k}(0)\\
 &=C_1 \sum_{k \leq -c t}\frac{e^{-t+(j-k)\varsigma(\frac{2t}{j-k})}}{\sqrt{2\pi}(1+4t^2+(j-k)^2)^{\frac{1}{4}}} u_{k}(0)\\
&\leq C_1  \frac{e^{-t}}{\sqrt{2\pi}(1+4t^2+j^2)^{\frac{1}{4}}} \sum_{k \leq 0}e^{(j-k)\varsigma(\frac{1}{c})} u_{k}(0)\\
&\leq C_1 \sup_{k\leq 0} \{u_{k}^{*}\} \frac{e^{-t}}{\sqrt{2\pi}(1+4t^2+j^2)^{\frac{1}{4}}} \sum_{k \leq 0}e^{(-k)\varsigma(\frac{1}{c})} \\
& = \tilde{C}_1 \frac{e^{-t}}{\sqrt{2\pi}(1+(4+c^2)t^2)^{\frac{1}{4}}}\\
& \leq \tilde{C}_1 e^{-t/2},
\end{align*}
where $\ds\tilde{C}_1 = C_1 \frac{e^{\varsigma(\frac{1}{c})}}{1-e^{\varsigma(\frac{1}{c})}}  \sup_{k\leq 0} \{u_{k}^{*}\}$.

Let $\ds-\sigma=-1+2 \frac{\varsigma(z_1)}{z_1}=\max_{-c t<k \leq 0} -1+2 \frac{\varsigma(z)}{z}$ with $z=\frac{2t}{j-k}$ and $z_1=\frac{2}{c}$. We remark that  $c \geq c^* \approx 2.073$ and $z_1  \leq \frac{2}{c^*} \approx 0.9648$. Thus, $$\ds\sigma=1-2 \frac{\varsigma(z_1)}{z_1} \geq 1-2 \frac{\varsigma(0.9648)}{0.9648} \approx 0.0355793>0.$$
Then, for $A_2$ and large $t$, with $z=\frac{2t}{j-k}$ and $\sigma$ above we have
\begin{align*}
 A_{2}:&= C_1 e^{t}\sum_{-c t<k \leq 0} F(2t,j-k) u_{k}(0)\\
 &=C_1 \sum_{-c t<k \leq 0}\frac{e^{-t+(j-k)\varsigma(\frac{2t}{j-k})}}{\sqrt{2\pi}(1+4t^2+(j-k)^2)^{\frac{1}{4}}} u_{k}(0)\\
  &=C_1 \sum_{-c t<k \leq 0}\frac{e^{(-1+2 \frac{\varsigma(z)}{z})t}}{\sqrt{2\pi}(1+4t^2+(j-k)^2)^{\frac{1}{4}}} u_{k}(0)\\
&\leq C_1  \frac{e^{-\sigma t}}{\sqrt{2\pi}(1+4t^2+j^2)^{\frac{1}{4}}} \sum_{-c t<k \leq 0}u_{k}(0)\\
&\leq C_1 \sup_{k\leq 0} \{u_{k}^{*}\} \frac{c t e^{-\sigma t} }{\sqrt{2\pi}(1+(4+c^2)t^2)^{\frac{1}{4}}}\\
 &\leq C_1 e^{-\frac{\sigma}{2} t}.
\end{align*}

For $A_3$, on $[1, (1-q\epsilon)k_n]$, $-1+|j-k|\varsigma(2t/|j-k|)$ obtains a maximum at $k=(1-q\epsilon) k_n$. Therefore, by Proposition \ref{Main-lm-up}, there exists a positive real $C_\delta$ such that  $u_k(0)\leq \frac{C_\delta}{C_1} e^{-(\mu-\delta)k}$ for $\delta>0$ and for $k_n=(c-\frac{2}{z_0})t$ and  $j=c t$, $j-k \geq j-(1-q\epsilon)k_n$ for $k \in [1,(1-q\epsilon)k_n]$. Then, letting $z=\frac{2t}{j-k}$ we have $z \leq \frac{2}{c-(1-q\epsilon)(c-\frac{2}{z_0})}=\frac{2}{\frac{2}{z_0}+q\epsilon(c-\frac{2}{z_0})}:=z_2<z_0$. 
By Lemma \ref{Au-lm4}(2), with $z_2 < z_0$ and $-1+2\frac{\varsigma(z_0)+\mu}{z_0}- \mu c=\lambda(\mu)- \mu c=0$, then we can let $$\tilde{\epsilon}_3= -(-1+2\frac{\varsigma(z_2)+\mu}{z_2}- \mu c)>0.$$ Thus, for $A_3$
\begin{align*}
A_3&=C_1e^{t}\sum_{1 \leq k \leq (1-q\epsilon) k_n} F(2t,j-k) u_{k}(0)\\
&\leq C_{\delta}\sum_{1 \leq k \leq (1-q\epsilon) k_n} F(2t,j-k) e^{-(\mu-\delta) k}\\
&= C_{\delta}e^{t}\sum_{1 \leq k \leq (1-q\epsilon) k_n} \frac{1}{\sqrt{2\pi}}\frac{exp[-t+(j-k)\varsigma(\frac{2t}{j-k})-(\mu-\delta)k]}{(1+4t^2+(j-k)^2)^{\frac{1}{4}}}\\
&\leq C_{\delta} \frac{1}{\sqrt{2\pi}}\frac{1}{(1+4t^2+(\frac{2}{z_0}t)^2)^{\frac{1}{4}}}\sum_{1 \leq k \leq (1-q\epsilon) k_n} exp[-t+(j-k)\varsigma(\frac{2t}{j-k})-(\mu-\delta)k]\\
&=C_{\delta} \frac{1}{\sqrt{2\pi}}\frac{1}{(1+4t^2+(\frac{2}{z_0}t)^2)^{\frac{1}{4}}}\sum_{1 \leq k \leq (1-q\epsilon) k_n} exp[(-1+2\frac{\varsigma(z)+\mu}{z})t-\mu j+\delta k]\\
&\leq C_{\delta} \frac{1}{\sqrt{2\pi}}\frac{1}{(1+4t^2+(\frac{2}{z_0} t)^2)^{\frac{1}{4}}}\sum_{1 \leq k \leq (1-q\epsilon) k_n} exp[(-1+2\frac{\varsigma(z_2)+\mu}{z_2})t-\mu j+\delta k]\\
&= C_{\delta} \frac{1}{\sqrt{2\pi}}\frac{1}{(1+4t^2+(\frac{2}{z_0} t)^2)^{\frac{1}{4}}} exp[(-1+2\frac{\varsigma(z_2)+\mu}{z_2}-\mu c) t]\sum_{1 \leq k \leq (1-q\epsilon) k_n} e^{\delta k}\\
&\leq C_{\delta} \frac{1}{\sqrt{2\pi}}\frac{1}{(1+4t^2+(\frac{2}{z_0} t)^2)^{\frac{1}{4}}} exp[(-1+2\frac{\varsigma(z_2)+\mu)}{z_2}-\mu c) t] (1-q\epsilon) k_n e^{\delta (1-q\epsilon) k_n}\\
&= C_{\delta} \frac{1}{\sqrt{2\pi}}\frac{ (1-q\epsilon)(c-\frac{2}{z_0})t}{(1+4t^2+(\frac{2}{z_0} t)^2)^{\frac{1}{4}}} exp[(-\tilde{\epsilon}_3+\delta(1-q\epsilon)(c-\frac{2}{z_0})) t]\\
&\leq C_{\delta} e^{-\epsilon_3 t},
\end{align*}
where $\epsilon_3 = \tilde{\epsilon}_3/2$ and choose $\delta$ such that $\delta(1-q\epsilon)(c-\frac{2}{z_0})<\tilde{\epsilon}_3/2$.

For $A_4$, with $k \in ((1-q\epsilon)k_n,(1+q\epsilon)k_n]$, recalling that \eqref{Harnack-assumption}, $u_{k_n}(0)\leq e^{-(\mu+\epsilon)k_n}$,
\begin{align*}
A_4&=C_1e^{t}\sum_{(1-q\epsilon) k_n < k \leq (1+ q\epsilon) k_n} F(2 t,j-k) u_{k}(0)\\
&=C_1\sum_{(1-q\epsilon) k_n < k \leq (1+ q\epsilon) k_n}\frac{1}{\sqrt{2\pi}}\frac{exp[-t+|j-k|\varsigma(2t/|j-k|)]}{(1+4t^2+(j-k)^2)^{\frac{1}{4}}}u_{k}(0)\\
&\leq C_1\sum_{(1-q\epsilon) k_n < k \leq (1+ q\epsilon) k_n} \frac{1}{\sqrt{2\pi}}\frac{exp[-t+|j-k|\varsigma(2t/|j-k|)]}{(1+4t^2+(j-k)^2)^{\frac{1}{4}}}e^{-(\mu+\epsilon)k_n}\\
&= C_1\sum_{(1-q\epsilon) k_n < k \leq (1+ q\epsilon) k_n} \frac{1}{\sqrt{2\pi}}\frac{exp[-t+|j-k|\varsigma(2t/|j-k|)-(\mu+\epsilon)k_n]}{(1+4t^2+(j-k)^2)^{\frac{1}{4}}}\\
&\leq C_1 \sum_{(1-q\epsilon) k_n < k \leq (1+ q\epsilon) k_n} exp[-t+|j-k|\varsigma(2t/|j-k|)-(\mu+\epsilon)k_n]\\
&= C_1 e^{-\epsilon k_n}\sum_{(1-q\epsilon) k_n < k \leq (1+ q\epsilon) k_n} exp[(-1+2\frac{\varsigma(z_2)+\mu}{z_2}-\mu c)t+\mu(k-k_n)]\\
&\leq C_1 e^{-\epsilon k_n}\sum_{(1-q\epsilon) k_n < k \leq (1+ q\epsilon) k_n} exp[(-1+2\frac{\varsigma(z_2)+\mu}{z_2}-\mu c)t+\mu q\epsilon k_n)]\\
&\leq C_1 e^{(-1+\mu q)\epsilon k_n}\sum_{(1-q\epsilon) k_n < k \leq (1+ q\epsilon) k_n} exp[(-1+2\frac{\varsigma(z_2)+\mu}{z_2}-\mu c)t)]\\
&\leq C_1e^{-(1-\mu q)\epsilon (c-\frac{2}{z_0})t}\\
&=C_1e^{-\epsilon_4t},
\end{align*}
where $\epsilon_4=(1-\mu q)\epsilon (c-\frac{2}{z_0})$ and choose $q<1/\mu$.

For $A_5$  with $k \in (1+q\epsilon)k_n, j)$, by Proposition \ref{Main-lm-up}, there exists a positive real $C_\delta$ such that  $u_k(0)\leq \frac{C_\delta}{C_1} e^{-(\mu-\delta)k}$ for $\delta>0$. Recall that $k_n=(c-\frac{2}{z_0})t$, $j=c t$, $j-k >0$ for $k \in ((1+q\epsilon)k_n, j)$. Then, letting $z=\frac{2t}{j-k}$, $z_3:=\frac{2}{c-(1+q\epsilon)(c-\frac{2}{z_0})}=\frac{2}{\frac{2}{z_0}-q\epsilon(c-\frac{2}{z_0})}>z_0$ for $\epsilon$ small and we have $z \geq z_3$ for $k \in ((1+q\epsilon)k_n, j)$. Let $\tilde{\epsilon}_5= -(-1+2\frac{\varsigma(z_3)+\mu}{z_3}- (\mu-\delta) c)$. Choose $\delta$ such that $\tilde{\epsilon}_5>0$, which can be done because of Lemma \ref{Au-lm4}(2) with $z_3 > z_0$ and $-1+2\frac{\varsigma(z_0)+\mu}{z_0}- \mu c=\lambda(\mu)- \mu c=0$. Then, letting $z=\frac{2t}{j-k}$ we have  
\begin{align*}
A_5&=C_1e^{t}\sum_{(1+q\epsilon)k_n < k < j} F(2t,j-k) u_{k}(0)\\
&\leq C_{\delta}\sum_{(1+q\epsilon) k_n< k < j} F(2t,j-k) e^{-(\mu-\delta) k}\\
&= C_{\delta}e^{t}\sum_{(1+q\epsilon) k_n< k < j} \frac{1}{\sqrt{2\pi}}\frac{exp[-t+(j-k)\varsigma(\frac{2t}{j-k})-(\mu-\delta)k]}{(1+4t^2+(j-k)^2)^{\frac{1}{4}}}\\
&\leq C_{\delta}\sum_{(1+q\epsilon) k_n< k < j} exp[-t+(j-k)\varsigma(\frac{2t}{j-k})-(\mu-\delta)k]\\
&=C_{\delta}\sum_{(1+q\epsilon) k_n< k < j} exp[(-1+2\frac{\varsigma(z)+\mu}{z})t-\mu j+\delta k]\\
&\leq C_{\delta} \sum_{(1+q\epsilon) k_n< k < j} exp[(-1+2\frac{\varsigma(z_3)+\mu}{z_3})t-\mu j+\delta k]\\
&\leq C_{\delta}\sum_{(1+q\epsilon) k_n< k < j} exp[((-1+2\frac{\varsigma(z_3)+\mu}{z_3})-(\mu-\delta)c)t]\\
&\leq C_{\delta} c t e^{-\tilde{\epsilon}_5 t}\\
&\leq C_{\delta} e^{-\epsilon_5 t},
\end{align*}
where $\epsilon_5=\tilde{\epsilon}_5/2$. 

By Proposition \ref{Main-lm-up}, there exists a positive real $C_\delta$ such that  $u_j(0)\leq \frac{C_\delta}{C_1} e^{-(\mu-\delta)j}$ for $\delta>0$. Let $\tilde{\epsilon}_6=-(1-(\mu-\delta)c)=\lambda(\mu)-1-\delta c$. $\delta$ can be chosen such that  $\tilde{\epsilon}_6>0$ since $\lambda(\mu)>1$ for $\mu>0$. For $A_6$, we have
\begin{align*}
A_6&= C_1 e^{t}\frac{1}{\sqrt{2\pi}}\frac{1}{(1+t^2)^{\frac{1}{4}}} u_{j}(0)\\
& \leq C_\delta e^{t}\frac{1}{\sqrt{2\pi}}\frac{1}{(1+t^2)^{\frac{1}{4}}} e^{-(\mu-\delta)j}\\
& \leq C_\delta \frac{1}{\sqrt{2\pi}}\frac{1}{(1+t^2)^{\frac{1}{4}}} e^{(1-(\mu-\delta)c)t}\\
&=  C_\delta \frac{1}{\sqrt{2\pi}}\frac{1}{(1+t^2)^{\frac{1}{4}}} e^{-\tilde{\epsilon}_6t}\\
&\leq C_{\delta} e^{-\epsilon_6 t},
\end{align*}
where $\epsilon_6=\tilde{\epsilon}_6/2$. 
Note that $A_7 < A_1+A_2+A_3+A_4+A_5$ for $u_{k}(0)$ decreasing in $k$ (If $u_{k}(0)\equiv const$, $A_7 = A_1+A_2+A_3+A_4+A_5$ by symmetry).

Finally, we have that 
\begin{align*}
u_j(t)&=A(t)-B(t)\\
& \leq 2( \tilde C_{1} e^{-\frac{t}{2}} +C_{1} e^{-\frac{\sigma}{2}t} +C_{\delta} e^{-\epsilon_3 t}+C_1 e^{-\epsilon_4 t}+C_{\delta} e^{-\epsilon_5 t}+C_{\delta} e^{-\epsilon_6 t})-C_B \sqrt{t} e^{-\hat \sigma_1 t-\hat\sigma_2}.
\end{align*}

As $\sigma_{1,2} \to 0$, $\hat \sigma_1 \to 0$. Thus, $\hat \sigma_1$ can be chosen such that $\hat \sigma_1<\min\{\frac{1}{2}, \frac{\sigma}{2}, \epsilon_3,\epsilon_4,\epsilon_5,\epsilon_6 \}$. Therefore, as $t$ is large enough, $A(t)-B(t)<0$, which causes a contradiction.
\end{proof}

\section{Nonexistence of Transition Fronts}
In this section, we shall investigate the conditions under which transition fronts do not exist. We shall prove part (2) of the main theorem (Theorem \ref{Existence and Nonexistence}) in the following proposition.
\begin{proposition}
\label{Nonexistence-prop}
Transition fronts do not exist under the following conditions:
\begin{itemize}
\item[(1)] $\lambda>\lambda^*$;
\item[(2)] for $\lambda \in [1,\lambda^*)$, either [i] $c<c^*$ or [ii] $c>\hat{c}$.
\end{itemize}
\end{proposition}

It is known that there is a minimal speed (spreading speed) $c^{*}$ such that transition fronts may exist, that is, transition fronts do not exist for $c<c^*$ (See Proposition \ref{lowerbound-lemma}). Thus, all $\mu$s of valid positive eigenvectors are located in $(0,\mu^*)$. However, if $\lambda \in (1,\lambda(\mu^*))$, due to lemma \ref{p_solution-lemma}, there are also no valid positive eigenvectors for $\mu \in (0,\hat{\mu})$. Proposition \ref{Nonexistence-prop} shows that there is a maximal speed $\hat{c}=\frac{\lambda}{\hat{\mu}}$ to prevent the existence of transition fronts, that is, transition fronts do not exist for $c> \hat{c}$ or $\mu<\hat{\mu}$. If $\lambda>\lambda(\mu^*)$, then $\hat{\mu}>\mu^*$ and there are none valid positive eigenvectors at all. The following Figure \ref{fig:existencefig} shows the existence intervals of transition fronts to \eqref{main-eq} with parameter values $(c,\mu) \in [c^*,\hat c] \times[\hat \mu,\mu^*]$. 
\begin{figure}[htp]
  \centering
  \includegraphics[width=100mm]{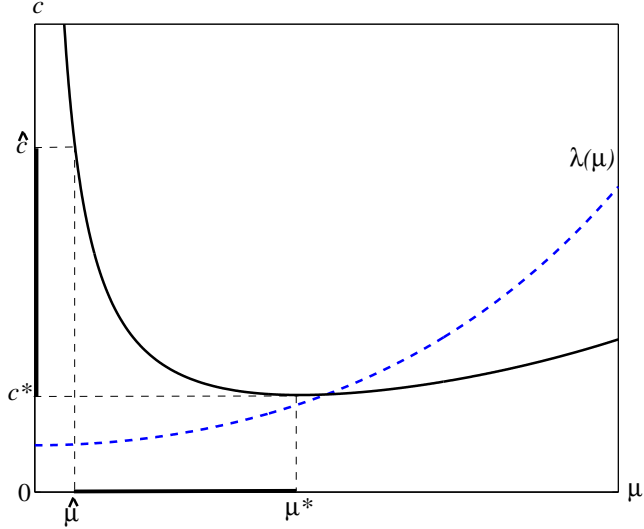}
        \caption{The solid curve is for speed auxiliary function $c(\mu)=\frac{\lambda(\mu)}{\mu}$. The dashed curve is for principal eigenvalue $\lambda(\mu)$. The parameter values of existence region of transition fronts are located on $[c^*,\hat{c}] \times [\hat{\mu},\mu^*]$, where $c^*=\ds\inf_{\mu>0} c(\mu)$ and $\hat{c}=c(\hat{\mu})$ with $\hat{\mu}$ satisfying $\lambda(\hat{\mu})=\lambda$.}
    \label{fig:existencefig}
 \end{figure}

From Figure \ref{fig:existencefig} and Propositions \ref{existence-prop1}, \ref{Main-lm-up} and \ref{main-lm4}, we have the following facts for transition fronts if they exist:
\begin{itemize}
\item[(1)] For any $\epsilon>0$, there exists a $T>0$ such that for $t>T$ and $j>ct$,$$C_1e^{-(\mu+\epsilon)(j-ct)} \leq u_j(t) \leq C_2e^{-(\mu-\epsilon)(j-ct)}\text{ (see Propositions \ref{Main-lm-up} - \ref{main-lm4}}).$$
\item[(2)] Due to the spreading properties of transition fronts, the lower bound of speed (minimal wave speed) is $c^*$, corresponding to the upper bound of $\mu$ (i.e. $\mu^*$). Then we must have $u_j(t) \geq K e^{-\mu^*(j-ct)}$ for $t$ large and some $K>0$.
\item[(3)] The upper bound of speed (maximal wave speed) is given by $\hat c$, corresponding to the lower bound of $\mu$ (i.e. $\hat\mu$). Thus, we must have $u_j(t) \leq K e^{-\hat{\mu}(j-ct)}$ for $t$ large and some $K>0$, that is controlled by the spectral bound $\lambda=\lambda(\hat{\mu})$.
\end{itemize}

We see that if $\lambda>\lambda^*$ and $\hat{\mu}>\mu^*$, this causes a contradiction of (2) and (3). If $c>\hat{c}$ and $\mu <\hat{\mu}$, this causes a contradiction of (1) and (3). 

\subsection{Spreading Speeds and the Lower Bound of Wave Speeds $c^*$} \label{lower bound section}
First, we shall show that $c^*$ is the lower bound of the speeds (minimal wave speed) in this subsection. For simplicity, we write $u_{j}(t)$ for $u_{j}(t;z)$ if no confusion occurs with the initial $z$. Define
\begin{equation}
\label{x-xi-space}
\hat{X}^+=\{v_{j}\geq 0\,|\,
 \liminf_{r\to -\infty}\inf_{j \leq r} v_{j}>0,
v_{j}=0\text{ for } j\in\ZZ\text{ with }j >N_0, \text{ for some } N_0>0\}.
\end{equation}

\begin{definition}
\label{spreading-speed-interval-def}
 A number $c^*$ is called the spreading speed of \eqref{main-eq} if for any $z \in \hat{X}^+$,
\begin{equation*}
\liminf_{j \leq ct, t\to
\infty} u_{j}(t)>0, \forall c<c^*,
\end{equation*}
and
\begin{equation*}
  \limsup_{j \geq ct, t\to \infty}u_{j}(t)=0,\forall c>c^*,
\end{equation*} where $u_{j}(0)=z_{j}$ is the initial condition.
\end{definition}

We remark that for homogenous and periodically heterogenous KPP-Fisher equations, the spreading speed exists. For \eqref{main-eq}, we have the following:
\begin{lemma}[See Theorem 2.2 in \cite{KoSh1,KoSh2}]
\label{spreading-prop1}
   The spreading speed of \eqref{main-eq} $c^*$ exists. Moreover, $c^*$ of \eqref{main-eq} with localized periodic inhomogeneity coincides with that of \eqref{main-eq} with corresponding periodic inhomogeneity.
\end{lemma}

\begin{lemma}[See Theorem 2.3 in \cite{KoSh1,KoSh2}]
\label{spreading-prop2}
   For each $\delta>0$, $r>0$, and $z \in X^+$ satisfying that $z_j \geq \delta$ for $|j| \leq r$,
   \begin{equation*}
\limsup_{|j| \leq ct, t\to
\infty} (u_{j}(t)-u_j^*)=0, \forall 0< c <c^*,
\end{equation*}
\end{lemma}

We remark that, in particular, for our main equation \eqref{main-eq}, the spreading speed coincides with the definition $c^*=\ds\frac{\lambda(\mu^*)}{\mu^*}=\inf_{\mu>0}\frac{\lambda(\mu)}{\mu}$ in the introduction of the current paper.

\begin{proposition}[Minimal Wave Speed]
\label{lowerbound-lemma}
   There does not exist a transition front of \eqref{main-eq} with speed less than $c^*$.
\end{proposition}
\begin{proof}
We prove this lemma by contradiction. Suppose that there is a transition front with speed $c <c^*$. Pick $c<c_1<c^*$. Choose $t_n$ such that $j_n=c_1t_n \in \ZZ$. By Lemma  \ref{spreading-prop1}, $\ds\liminf_{j_n \leq c_1t_n, t_n\to \infty} u_{j_n}(t_n)>0$. On the other hand, $j_n-ct_n=(c_1-c)t_n \to \infty$, by the definition of transition front, $\ds\lim_{j_n-ct_n \to \infty}u_{j_n}(t_n)=0$, which causes a contradiction.
\end{proof}

\subsection{Nonexistence of Transition Fronts for $\lambda>\lambda^*$}
In this subsection, we will show that if $\lambda>\lambda(\mu^{*})$, there are no transition fronts. In biological sense, transition fronts won't exist in strongly localized spatial inhomogeneous environments. We shall prove the following proposition.
\begin{proposition}
\label{main-lm}
If $\lambda>\lambda(\mu^{*})$, any  entire solution $u_j(t)$ of
  \eqref{main-eq} such that $0<u_j(t)<u_j^*$ satisfies that for any $c<\hat{c}$, there exists a $K>0$ such that for all $(t,j) \in \RR^{-}\times \ZZ$, $$ u_j(t) \leq K e^{-\mu^{*}(|j|-ct)},$$
  where $\mu^{*}$ is such that  $c^{*}=\frac{\lambda(\mu^{*})}{\mu^{*}}=\ds\inf_{\mu>0}\frac{\lambda(\mu)}{\mu}$. In particular, no transition fronts exist if $\lambda>\lambda^*$.
\end{proposition}

To prove the Proposition \ref{main-lm}, we show the following Lemmas \ref{Aux-comparison-lemma} and \ref{main-lm2}.
\begin{lemma}
\label{Aux-comparison-lemma}
For each $m \in \ZZ$ and $\epsilon > 0$ there exist $k_\epsilon,\delta > 0$ such that if $u_j(t)$ solves \eqref{main-eq} with
$u_{j_0}(0) \geq \gamma$, for any given $j_0$ and $\gamma \leq \frac{\delta}{2}$, then for $t \geq 0$ and $j \leq j_0+m-c^* t$,
$$
u_j(t)\geq k_\epsilon \gamma e^{(1-\epsilon)t} F(2t, j-j_0).
$$
\end{lemma}
\begin{proof}
Without loss of generality, set $j_0=0$. Let $l=\ds\min_j \{a_j\} \leq 1$. Note that $v_j(t)=\gamma k_\epsilon e^{(l-\epsilon)t} h^{\ZZ}_{2t}(j) v_0(0)$ is a solution of $$\dot v_{j}(t)=v_{j+1}(t)-2v_{j}(t)+v_{j-1}(t)+(l-\epsilon) v_{j}(t).$$ 
Since $\|v(t)\|_\infty \leq \gamma k_\epsilon e^{(1-\epsilon)t}$, we have $\|v(t)\|< \gamma$ if $k_\epsilon=e^{-2 t_\epsilon}$ for $t \leq t_\epsilon$. Let $\bar j=-c^* t+2 m$. Since by Lemma \ref{Au-lm4} (2) and (4), $c^*$ satisfies $-1 +2\frac{\zeta(2/c^*)}{2/c^*} =0$, there is a $t_\epsilon$ such that for $t>t_\epsilon$, 
\begin{align*}
&-1+2\frac{\zeta(2t/|\bar j|)}{2t/|\bar j|} \\
&=-1+2\frac{\zeta(2t/|c^* t-2 m|)}{2t/|c^* t-2 m|} \\
&=-1+2\frac{\zeta(2/|c^* -2 \frac{m}{t}|)}{2/|c^* -2 \frac{m}{t}|} \\
&\leq-1+2\frac{\zeta(2/|c^* -2 \frac{|m|}{t_\epsilon}|)}{2/|c^* -2 \frac{|m|}{t_\epsilon}|} \\
&< \epsilon/2.
\end{align*}
Then for $t>t_\epsilon$, $-(1+\epsilon)+2\frac{\zeta(2t/|\bar j|)}{2t/|\bar j|}<-\epsilon/2$. Therefore, for $t>t_\epsilon$, with \eqref{heatkernel}
\begin{align*}
v_{\bar j}(t)&\leq \gamma k_\epsilon e^{(1-\epsilon) t} h^{\ZZ}_{2 t}(\bar j) v_0(0)\\
& \leq (1+\epsilon)\gamma k_\epsilon v_0(0) e^{(-(1+\epsilon)+2\frac{\zeta(2t/|\bar j|)}{2t/|\bar j|})t}\\
&\leq (1+\epsilon)\gamma k_\epsilon v_0(0).
\end{align*}
Then for $t \geq t_\epsilon$, replace $k_\epsilon$ with $\frac{1}{(1+2\epsilon)v_0(0) }$ if the original $k_\epsilon$ is bigger than $\frac{1}{(1+2\epsilon)v_0(0) }$, and thus $(1+\epsilon) k_\epsilon v_0(0) <1$, that is, we also have $v_{\bar j}(t) \leq \gamma$. Furthermore, for $t>t_\epsilon \geq \frac{2|m|}{c^*}$ such that $\bar j< 0$, $v_{j}(t) \leq \gamma$ holds for all $j \leq \bar j$, because $v_{j}(t) \leq  \gamma (1+\epsilon)k_\epsilon e^{(1-\epsilon)t} F(2t,j) v_0(0)$ and $F(2t,j)$ is increasing in $j \in (-\infty, \bar j)$ by Lemma \ref{Au-lm4} (5). Thus we have either
$$
v_{j}(t) \leq \gamma, \forall t \in [0,t_\epsilon), j \in \ZZ,
$$
or 
$$
v_{j}(t)\leq \gamma, \forall t \geq t_\epsilon, j < -c^* t+2m.
$$

Let $\Omega=\big\{(t,j) \in \RR \times \ZZ| t \in [0,t_\epsilon) \times \ZZ \cup [t_\epsilon, \infty) \times (-\infty, -c^* t+2m)\big\}$. 
By spreading properties in Lemma \ref{spreading-prop2}, for any given $0<\gamma<\min \{u_j^*\}$, there exists a $t_\epsilon$ (if necessary, replace the original $t_\epsilon$ with the larger number) such that $u_{\bar j}(t) \geq \gamma$ and thus $v_{\bar j}(t) \leq \gamma \leq u_{\bar j}(t)$. Moreover, for any given $\epsilon$, there exists a $\gamma$ such that $f_j(\gamma)-l>0$ and thus $v$ is a sub-solution of \eqref{main-eq} on $\Omega$. Indeed,
\begin{align*}
&\dot v_{j}(t)-[v_{j+1}(t)-2v_{j}(t)+v_{j-1}(t)+f(v_{j}(t))v_{j}(t)]\\
&=\dot v_{j}(t)-[v_{j+1}(t)-2v_{j}(t)+v_{j-1}(t)+(l-\epsilon) v_{j}(t)]+(-f(v_{j}(t))+l-\epsilon)v_{j}(t)\\
&=(-f(v_{j}(t))+l-\epsilon)v_{j}(t)\\
&\leq(-f(\gamma)+l-\epsilon)v_{j}(t)\\
&\leq 0.
\end{align*}
Since $v_j(0) \leq u_j(0)$ for $j \in \ZZ$, by \textbf{comparison principle} (Lemma 2.1 in \cite{CGW}), $v_j(t) \leq u_j(t)$ on $\bar \Omega$. Note that $t \geq 0$ and $j \leq j_0+m-c^* t$ is a subset of $\bar \Omega$ and this completes the proof.

\end{proof}

\begin{lemma}
\label{main-lm2}
For every $\epsilon \in (0,1)$, there exists a $K_{\epsilon}>1$ such that 
$$ u_j(t) \leq K_{\epsilon}u_{0}(0) \sqrt{|t|}e^{\hat{\mu} j+ (\lambda_M-\epsilon)t},$$ for all $t\leq -1$ and $j \in [M,- c_{\epsilon} t]$, with $c_{\epsilon}=\frac{\lambda_M-\epsilon}{\hat{\mu}}$.
\end{lemma}
\begin{proof}
Suppose the lemma to be false. Then there exist $\bar t\leq-1$ and $j_0 \in [M,-c_{\epsilon} \bar t]$ such that 
\begin{equation}
\label{Aux-sub-eq-j0}
 u_{j_0}(\bar t) \geq K_{\epsilon}u_{0}(0) \sqrt{|\bar t|}e^{\hat{\mu}( j_0+c_{\epsilon}\bar t)}.
 \end{equation}
Let $\tilde u=\delta \phi^{(M)}$ be a sub-solution of \eqref{main-eq} as in Lemma \ref{Truncated-sub}, where $\phi^{(M)}$ is the principal eigenvector to \eqref{eigen-eq-trucated} with $\|\phi^{(M)}\|_{\infty}=1$. Let $\underline{v}_j$ be given by
\begin{equation}
\label{Aux-sub-eq-v}
\underline{v}_j=\min\{\delta,A e^{(\lambda_M-\epsilon) t}\}\phi_j^{(M)}
\end{equation}
 Thus there exists a $\delta$ such that \underline{v} is also a sub-solution of \eqref{main-eq} for any $A>0$. Indeed, if $\tilde v_j:=A e^{(\lambda_M-\epsilon) t }\phi_j^{(M)}<\delta\phi_j^{(M)}$, choosing a $\delta$ such that $f_j(0)-f_j(\delta)<\epsilon$, we have
\begin{align*}
&(\tilde v_{j})_{t}- [\tilde v_{j+1}-2\tilde v_{j}+\tilde v_{j-1}+f_{j}(\tilde v_{j})\tilde v_{j}]\\
= &(\tilde v_{j})_{t}- \tilde v_{j+1}-2\tilde v_{j}+\tilde u_{j-1}+a_{j} \tilde v_{j}]+ (f_j(0)-f_j(\tilde v_{j}))\tilde v_{j}\\
=&(\lambda_M-\epsilon-\lambda_M +(f_j(0)-f_j(\tilde v_{j})))\tilde v_{j}\\
=&(-\epsilon +(f_j(0)-f_j(\tilde v_{j})))\tilde v_{j}\\
\leq &(-\epsilon+(f_j(0)-f_j(\delta)))\tilde v_{j}\\
\leq& 0.
\end{align*}
For $\tilde v \geq \delta\phi^{(M)}$, the above inequality holds for $\underline{v}_j=\tilde u_j$ by the calculation in Lemma \ref{Truncated-sub}.

With a possible translation, we assume that $u_0(0) < \tilde u_0$ and $f_j(0)>1$ for $j=0$. 
 Let $\beta$ be chosen later such that $0<\beta<1$.
For $-M \leq j \leq M$, let $z_1=2\beta |\bar t|/|j-j_0|$ and $z=2\beta |\bar t|/j_0$, we have $z_1> z$ and then by the monotonicity of $\frac{\varsigma(z)}{z}$ in Lemma \ref{Au-lm4}, we have $\frac{\varsigma(z_1)}{z_1}>\frac{\varsigma(z)}{z}$. By \textbf{Heat Kernel Estimate} \eqref{heatkernel} $h^{\ZZ}_{t}(j) \asymp F(t,j)$, there exist positive $C_{1}$ and $C_{2}$ such that $C_{1}F(t,j) \leq h^{\ZZ}_{t}(j)\leq  C_{2}F(t,j)$. Therefore, together with \eqref{Aux-sub-eq-j0} and Lemma \ref{Aux-comparison-lemma}, we have
\begin{align*}
u_j(\bar t&+\beta|\bar t|)
\geq C_{1} e^{(1-\epsilon)\beta|\bar t|}\sum_{k}F(2\beta|\bar t|,j-k)u_k(\bar t)\\
&\geq C_{1} e^{(1-\epsilon)\beta|\bar t|}F(2\beta|\bar t|,j-j_0)u_{j_0}(\bar t)\\
&\geq C_{1} e^{(1-\epsilon)\beta|\bar t|}F(2\beta|\bar t|,j-j_0)K_{\epsilon}u_{0}(0) \sqrt{|\bar t|}e^{\hat{\mu} j_0+  (\lambda_M-\epsilon) \bar t}\\
&= C_{1} e^{(1-\epsilon)\beta |t|} \frac{1}{\sqrt{2\pi}}\frac{exp[-2\beta|\bar t|+|j-j_0|\varsigma(2\beta |\bar t|/|j-j_0|)]}{(1+4\beta^2\bar t^2+(j-j_0)^2)^{\frac{1}{4}}} K_{\epsilon}u_{0}(0) \sqrt{|\bar t|}e^{\hat{\mu} j_0+  (\lambda_M-\epsilon) \bar t}\\
&= C_{1} K_{\epsilon}u_{0}(0) \frac{1}{\sqrt{2\pi}}\frac{exp[(-2+2\frac{\varsigma(z_1)}{z_1})\beta|\bar t|]}{(1+4\beta^2\bar t^2+(j-j_0)^2)^{\frac{1}{4}}}  \sqrt{|\bar t|}e^{(1-\epsilon)\beta |t|+\hat{\mu} j_0+  (\lambda_M-\epsilon) \bar t}\\
&\geq C_{1} K_{\epsilon}u_{0}(0) \frac{1}{\sqrt{2\pi}}\frac{exp[(-2+2\frac{\varsigma(z)}{z})\beta|\bar t|]}{(1+4\beta^2\bar t^2+j_0^2)^{\frac{1}{4}}}  \sqrt{|\bar t|}e^{(1-\epsilon)\beta |t|+\hat{\mu} j_0+  (\lambda_M-\epsilon) \bar t}\\
&\geq  C_{1} K_{\epsilon}u_{0}(0) \frac{1}{\sqrt{2\pi}}\frac{exp[(-2+2\frac{\varsigma(z)}{z})\beta|\bar t|]}{(1+4\bar t^2+(c_\epsilon \bar t)^2)^{\frac{1}{4}}}  \sqrt{|\bar t|}e^{(1-\epsilon)\beta |t|+\hat{\mu} j_0+  (\lambda_M-\epsilon) \bar t}\\
&\geq  C_{1} K_{\epsilon}u_{0}(0) \frac{1}{\sqrt{2\pi}}\frac{exp[(-2+2\frac{\varsigma(z)}{z})\beta|\bar t|]}{(5+c_\epsilon^2)^{\frac{1}{4}}} e^{(1-\epsilon)\beta |t|+\hat{\mu} j_0+  (\lambda_M-\epsilon) \bar t}\\
&\geq K'_{\epsilon}u_{0}(0) exp\{(-2+2\frac{\varsigma(z)}{z}+(1-\epsilon))\beta|\bar t|+\hat{\mu} j_0+  (\lambda_M-\epsilon) \bar t\}\\
&:=K''_{\epsilon},
\end{align*}
where $K'_{\epsilon}=C_{1} K_{\epsilon}\frac{1}{\sqrt{2\pi}(5+c_\epsilon^2)^{\frac{1}{4}}}$.
Thus, with \textbf{Comparison Principle}, choosing $A=K''_{\epsilon}$ in \eqref{Aux-sub-eq-v}, we have
$$u_j(t+\bar t+\beta|\bar t|)\geq \underline{v}_j(t)$$ for $t>0$. In particular, letting $t=(1-\beta)|\bar t|$, we have 
$$
u_j(0)\geq  \min\big\{\delta,K''_{\epsilon} e^{(\lambda_M-\epsilon) (1-\beta)|\bar t|}\big\}\phi_j^{(M)}.
$$
By choosing $\beta$ such that $K''_{\epsilon}e^{(\lambda_M-\epsilon) (1-\beta)|\bar t|}\phi^{(M)}=K'_{\epsilon}u_{0}(0)$, that is, 
\begin{equation*}
exp\big\{(-2+2\frac{\varsigma(z)}{z}+(1-\epsilon))\beta |\bar t|+\hat{\mu} j_0+  (\lambda_M-\epsilon) \bar t\big\} \times e^{(\lambda_M-\epsilon) (1-\beta)|\bar t|}=1.
\end{equation*}
Therefore
\begin{equation*}
\label{beta-eq}
(-2+2\frac{\varsigma(z)}{z}+(1-\epsilon)-(\lambda_M-\epsilon))\beta |\bar t|+\hat{\mu} j_0=0.
\end{equation*}
Recalling that $z=2\beta |\bar t|/j_0$, $-2+2\frac{\varsigma(z)}{z}+(1-\epsilon)-(\lambda_M-\epsilon))\beta |\bar t|+ 2\hat{\mu} \beta |\bar t|/z=0,$ that is,
$$-1+2\frac{\varsigma(z)+\hat\mu}{z}-(\lambda_M-\epsilon)-\epsilon=0.$$
Thus, let $g(z)$ be as in Lemma \ref{Au-lm4} and we have
$$
g(z)=\lambda_M< \lambda=\lambda(\hat \mu).
$$

By the concavity of $g(z)$ in Lemma \ref{Au-lm4} (2), and $M>>1$, there exists at least one $\bar z< z_0=csch(\hat\mu)$ such that $g(\bar z)=0$.
Recall that $\hat c< \frac{2}{z_0}$ for $\hat u> \mu^*$ by Lemma \ref{Au-lm4} (4). With $c_\epsilon<\hat c$, $j_0 \leq c_\epsilon |\bar t|$ and $\bar z=2\beta |\bar t|/j_0$, we have
\begin{equation*}
\beta=\frac{\bar z j_0}{2|\bar t|}
\leq \frac{z_0 j_0}{2|\bar t|}
\leq \frac{z_0 c_\epsilon}{2}
\leq \frac{z_0 \hat c}{2}
<1.
\end{equation*}
Thus, $\beta<1$ as required.
Finally, by taking $K_{\epsilon}$ large enough and $j=0$, $$ u_0(0) \geq \min\{\delta \phi_0^{(M)}, K'_{\epsilon} u_0(0)\} \geq \min\{\delta \phi_0^{(M)}, 2 u_0(0)\},$$ which causes a contradiction. This completes the proof.
\end{proof}
\begin{remark}
\label{rk-main-lm2}
Lemma \ref{main-lm2} holds for fixed $j$ and so for any $j$ on a compact set without the assumption $\lambda > \lambda^*$. Indeed, in this case, we can remove the restriction of $c_\epsilon=\frac{\lambda_M-\epsilon}{\hat \mu}$ and freely choose $c_\epsilon<\frac{2}{z_0}$ in the lemma.
\end{remark}

\begin{lemma}
\label{main-lm1}
Assume that  $c,c_1 \in (c^{*},\hat{c})$ with $c<c_1$,  there exists a $K_{0}>0$ and $\tau>0$ such that $$ u_j(t) \leq K_{0}u_{0}(0) e^{ \mu^{*}(j+ct)},$$ for all $(t,j) \in (-\infty,-1)\times [M, -c_1t]$ as well as $(t,j) \in (-\infty,-t_0)\times [M, \infty)$.
\end{lemma}

\begin{proof}
Pick $\epsilon>0$ such that $c_{\epsilon}=c_1$. Note that $\lambda(\hat \mu)=\lambda>\lambda^*=\lambda(\mu^*)$ implies that $\hat \mu> \mu^*$. By Lemma \ref{main-lm2}, there is $C_{0}>0$ such that for all $t\leq-1$ and $j \in [M,-c_{\epsilon}t]$,
$$ u_j(t) \leq K_{\epsilon}u_{0}(0) \sqrt{|t|}e^{\hat{\mu} (j+c_1t)}\leq K_{\epsilon}u_{0}(0) \sqrt{|t|}e^{\mu^{*}(j+c_1t)}\leq K_{0}u_{0}(0)e^{\mu^{*}(j+ct)}.$$
We can let $\ds t_{0}\equiv \frac{ln(K_{0}u_{0}(0))- ln (\ds\max_{j}u_j^*)}{\mu^*(c_1-c)}>0$ to complete the proof of the second part. Indeed, for $t < -t_0$, we have $K_{0}u_{0}(0) e^{ \mu^{*}(j+ct)} \geq u_j^*$ for all $j> -c_1t$. Since $u_j(t) \leq u_j^*$ for all $(t,j) \in \RR \times \ZZ$, the inequality holds for all $(t,j) \in (-\infty,-t_0)\times [M, \infty)$.
\end{proof}

\begin{proof}[Proof of Proposition \ref{main-lm}]
Given $c,c_1 \in (c^{*},\hat{c})$ with $c<c_1$. Let $\tau_{1}\equiv M/c_1$ and so $M \leq -c_1 t$ for $t \leq -\tau_1$. By Lemma \ref{main-lm1}, for $t\leq -\tau_{1}$,
$$ u_{M}(t) \leq K_{0}u_{0}(0) e^{\mu^{*}(M+ct)}.$$
Next, for $t\leq -\tau_{0}$, we let
$$ v_{j}(t;t_{0}) \equiv K_{0}u_{0}(0) [e^{\mu^{*}(j+ct_{0}+c^{*}(t-t_{0}))}+e^{\mu^{*}(2M-j+ct)}].$$ Then $v_{j}(t;t_{0})$ is a super-solution on $(t_{0},\infty)\times(M,\infty)$. Moreover, for $t\leq -\tau_{0}$ and $j>M$, we have
$$ u_{j}(t_{0})\leq K_{0}u_{0}(0) e^{\mu^{*}(j+ct_{0})}\leq v_{j}(t_{0};t_{0}).$$ Since $c>c^{*}$, we have $u_{M}(t)\leq v_{M}(t;t_{0})$ for all $t \in (t_{0},-\tau_{1})$.  By comparison principle, $u_{j}(t)\leq v_{j}(t;t_{0})$ for all $t \in [t_{0},-\tau_{1}]$ and $j\geq M$. Letting $t_{0} \to -\infty$, we have for all $t \leq -\tau_{1}$ and $j\geq M$, $$ u_{j}(t)\leq K_{0}u_{0}(0) e^{\mu^{*}(2M-j+ct)}.$$ Similarly, we have for all $t \leq -\tau_{1}$ and $j\leq -M$, $$ u_{j}(t)\leq K_{0}u_{0}(0) e^{\mu^{*}(2M+j+ct)}.$$ Thus, for all $t \leq -\tau_{1}$ and $j\leq \ZZ \setminus (-M,M)$, $$ u_{j}(t)\leq K_{0}e^{2\mu^*M}u_{0}(0) e^{-\mu^{*}(|j|-ct)}.$$
The Harnack inequality extends this bound to all $t\leq -\tau_{1}-1$ and $j \in \ZZ$: $$ u_{j}(t)\leq K_{1}u_{0}(0) e^{-\mu^{*}(|j|-ct)}.$$ Thus,
$$ u_{j}(t)\leq K_{1}u_{0}(0) e^{-\mu^{*}(|j|-c(-\tau_{1}-1))+(1+\|a\|_{\infty})(t-(-\tau_{1}-1))},$$ for $t \geq -\tau_{1}-1$, where $\|a\|_{\infty}=\ds \max_{j} a_j$. We note that the right-hand side is a super solution. Thus, for $t\leq 0$ and $j \in \ZZ$, we have $$ u_{j}(t)\leq K_{2}u_{0}(0) e^{-\mu^{*}(|j|-ct)}.$$
Finally, we have the non-existence of transition fronts if $\lambda>\lambda^*$, because $\ds\lim_{j \to -\infty}u_j(t)=0$ under the above inequality.
\end{proof}

\subsection{The Upper Bound of Wave Speeds $\hat c$} \label{upper bound section}
Finally, we shall show $\hat{c}$ is the upper bound of the speeds (maximal wave speed) by investigating the nonexistence of transition fronts to \eqref{main-eq} for $c > \hat{c}$, which is corresponding to $\mu \in (0, \hat{\mu})$ where no valid positive eigenvectors of \eqref{eigen-eq} can be located.

\begin{lemma}
\label{main-lm-prop}
For all $\epsilon>0$, there exists $K_\epsilon>0$ such that
\begin{equation}
\label{ubounded0}
u_{j}(t) \leq K_{\epsilon} e^{\lambda(\hat{\mu}-\epsilon)t-(\hat{\mu}-\epsilon)j},\text{\ for all\ }  j\geq 0 \text{\ and \  } t \leq 0.
\end{equation}

\end{lemma}
\begin{proof}
First, there exist $M$ and $\bar \epsilon$ such that $\lambda(\hat{\mu}-\epsilon)=\lambda_M-\bar \epsilon$. By Remark \ref{rk-main-lm2} of Lemma \ref{main-lm2}, we have that \eqref{ubounded0} holds for fixed $j$ and thus also for $j$ in a bounded set of $\ZZ^+$. 
Second, we show that
\begin{equation}
\label{ubounded}
u_{j}(t) \leq C e^{t} \text{\ for all\ }  j\geq 0 \text{\ and \  } t \leq 0. 
\end{equation}

Let $\rho(t)=\ds \sum_{j \geq M_0} u_j(t)$, which is well-defined due to Proposition \ref{Main-lm-up}. Then
\begin{align*}
\dot \rho(t)&=\sum_{j \geq M_0} \dot u_j(t)\\
&= \sum_{j \geq M_0} (u_{j+1}-2u_{j}+u_{j-1}+f_j(u_{j})u_{j}(t))\\
&=  u_{M_0-1}(t)-u_{M_0}(t) + \sum_{j \geq M_0} (f_j(u_{j})u_{j}(t))
\end{align*}

Therefore, 
\begin{align*}
\dot \rho(t)-\rho(t)=  u_{M_0-1}(t)-u_{M_0}(t) + \sum_{j \geq M_0} (f_j(u_{j})-1)u_{j}(t).
\end{align*}

For $j>M_0$ and $t << -1$, $f_j(u_{j})-1=f_j(u_{j})-f_j(0)<0$ and thus $$\ds\sum_{j \geq M_0}(f_j(u_{j})-1)u_{j}(t)< 0.$$ 

Then,
\begin{align*}
\frac{d}{dt}(-e^{-t} \rho(t))&=-e^{-t}(\dot \rho(t)-\rho(t))\\
&= -e^{-t}( u_{M_0-1}(t)-u_{M_0}(t) + \sum_{j \geq M_0} (f_j(u_{j})-1)u_{j}(t))\\
&\leq -e^{-t}( u_{M_0-1}(t)-u_{M_0}(t)) \\
&\leq e^{-t}(u_{M_0}(t) +u_{M_0-1}(t))\\
&\leq e^{-t}K_{\epsilon} (1+e^{(\hat{\mu}-\epsilon)})e^{\lambda t-(\hat{\mu}-\epsilon)M_0} \\
&= K_{\epsilon} (1+e^{(\hat{\mu}-\epsilon)})e^{(\lambda -1)t-(\hat{\mu}-\epsilon)M_0}.
\end{align*}

Integrate both sides from $t(\leq 0)$ to 0, and we have $e^{-t} \rho(t)- \rho(0) \leq \frac{K_{\epsilon}(1+e^{(\hat{\mu}-\epsilon)})}{(\lambda -1)} e^{(\hat{\mu}-\epsilon)M_0}.$ 
Let $C=\rho(0)+\frac{K_{\epsilon}(1+e^{(\hat{\mu}-\epsilon)})}{(\lambda -1)} e^{(\hat{\mu}-\epsilon)M_0}$. For $t \leq 0$, $\rho(t) \leq C e^t$. 
Therefore, \eqref{ubounded} holds for $j \geq M_0$. We let 
\begin{equation}
\label{aux_4_12}
w_{j}(t) = e^{-t}u_{j}(t)- K_{\epsilon}e^{(\lambda(\hat{\mu}-\epsilon)-1)t-(\hat{\mu}-\epsilon)(j-M-1)}.
\end{equation}

Then, for $j>N$, we have
\begin{align*}
\dot{w}_{j}(t) &= e^{-t}\dot u_{j}(t) -e^{-t}u_{j}(t)-(\lambda(\hat{\mu}-\epsilon)-1)K_{\epsilon}e^{(\lambda(\hat{\mu}-\epsilon)-1)t-(\hat{\mu}-\epsilon)(j-M-1)}\\
&= e^{-t}(u_{j+1}-2u_{j}+u_{j-1}+(f_j(u_{j})-1)u_{j}(t))\\
&\quad\quad\quad\quad\quad\quad\quad\quad\quad\quad-(\lambda(\hat{\mu}-\epsilon)-1)K_{\epsilon}e^{(\lambda(\hat{\mu}-\epsilon)-1)t-(\hat{\mu}-\epsilon)(j-M-1)}\\
&= e^{-t}(u_{j+1}-2u_{j}+u_{j-1})-(\lambda(\hat{\mu}-\epsilon)-1)K_{\epsilon}e^{(\lambda(\hat{\mu}-\epsilon)-1)t-(\hat{\mu}-\epsilon)(j-M-1)}\\
&\quad\quad\quad\quad\quad\quad\quad\quad\quad\quad+e^{-t}(f_j(u_{j})-1)u_{j}(t).
\end{align*}
On the other hand, we have
\begin{align*}
w_{j+1}-2w_{j}+w_{j-1} &=e^{-t}(u_{j+1}-2u_{j}+u_{j-1})-(\lambda(\hat{\mu}-\epsilon)-1)K_{\epsilon}e^{(\lambda(\hat{\mu}-\epsilon)-1)t-(\hat{\mu}-\epsilon)(j-M-1)}
\end{align*}
Thus,
\begin{align}
\begin{split}
\label{aux_4_13}
\dot{w}_{j}(t)&-(w_{j+1}-2w_{j}+w_{j-1}) \\
&=e^{-t}(f_j(u_{j})-1)u_{j}(t)\\
&\leq 0.
\end{split}
\end{align}

Note that $w_{M_0}(t) \leq 0$ for $t \leq 0$ because \eqref{ubounded0} holds for fixed $j=M_0$ that has been proved previously. By \eqref{ubounded}, $ w_{j}(t) <  e^{-t}u_{j}(t)<C$ for all $j \geq 0$ and $t \leq 0$. For $t \leq 0$ and $j \geq M_0$, choose $\epsilon$ such that  $\lambda(\hat{\mu}-\epsilon)>1$, then $w_j(t)$ is bounded. Furthermore we claim that for $j>M_0$ and $t \leq 0$, $w_j(t)$ cannot attain a positive maximum, and there cannot be a sequence $(t_n,j_n)$ such that $w_{j_n}(t_n)$ tends to a positive supremum. Suppose that it obtains a positive maximum at $(t_0,j_0)$ and for $M_0 \leq j <j_0$, $w_j(t_0)<w_{j_0}(t_0)$. Then $\dot{w}_{j_0}(t_0)-(w_{j_0+1}-2w_{j_0}+w_{j_0-1}) >0$, which contradicts \eqref{aux_4_13}. Suppose that there is a sequence $(t_n,j_n)$ such that $w_{j_n}(t_n)$ tends to a positive supremum for $j>M_0$. Then $j_n \to \infty$ as $n \to \infty$, otherwise $j_n $ goes to some fixed $\hat j$ as $n \to \infty$ that contradicts with \eqref{ubounded0} holds for fixed $\hat j$. And then $t \to -\infty$ as $n \to \infty$. Otherwise, $t \to -T$ as $n \to \infty$ for some $T>0$. With \eqref{aux_4_12}, we have $\ds \lim_{n \to \infty} w_{j_n}(T)=0$. Therefore, $w_{j}(t)\leq0$ for all $j \geq M_0$ and $t \leq 0$, which implies that \eqref{ubounded0} holds for all $j \geq M_0$ and $t \leq 0$. Finally, with that \eqref{ubounded0} holds on the compact set, in particular  $[0,M_0]$, this completes the proof. 
\end{proof}

\begin{proof}[Proof of Proposition \ref{Nonexistence-prop}]
First, we proved the case with $\lambda>\lambda^*$ by Proposition \ref{main-lm}. Next, it has been shown in Proposition \ref{lowerbound-lemma} that there are no transition fronts for $c<c^*$ due to the properties of spreading speeds. Finally we prove the nonexistence for $c>\hat{c}$. Assume there exists a transition front $u_{j}(t)$ for $c>\hat{c}$. Let $\mu$ and $\hat \mu$ be such that $c=\frac{\lambda(\mu)}{\mu}$ and $\hat c=\frac{\lambda( \hat \mu)}{\hat \mu}$. Then we have that $0<\mu<\hat \mu<\mu^*$. By Proposition \ref{main-lm4}, we have
\begin{equation}
\label{u_bounded_estimate}
u_{j}(t) \ge K_{\epsilon} e^{-(\mu+\epsilon)(j-ct)},\text{\ for all\ }  j \geq ct+M_{0}  \text{\ and \  }t\geq 0. 
\end{equation}
By Lemma \ref{main-lm-prop}, in particular for $t=0$, we have that
\begin{equation}
\label{ubounded_t0}
u_{j}(0) \leq K_{\epsilon} e^{-(\hat{\mu}-\epsilon)j},\text{\ for all\ }  j\geq 0.
\end{equation}

Consider the linear periodic equation restricted on $[M_0,M_0+p]$ for $p>>1$ (i.e. $p$ is as large as required), that is, $v_{j+p}=v_{j}$ for any $j \in \ZZ$.
\begin{equation}
\label{partial-eq}
\ds\dot{v}_{j}=v_{j+1}-v_{j}+v_{j-1},\quad M_0 \leq j \leq M_0+p.
\end{equation}
Let $\bar v_j(t)=K_{\epsilon} e^{-(\hat{\mu}-\epsilon)(j-ct)}$ for $j \in [M_0,M_0+p)$ and $ \underline{v}_j(t)=u_j(t)$ for $j \in [M_0,M_0+p)$. Then by \eqref{ubounded_t0}
$\underline{v}_{j}(0) \leq \bar v_j(0)$.

By direct calculation, we have
\begin{align*}
&(\bar v_{j})_{t}- [\bar v_{j+1}-\bar v_{j}+\bar v_{j-1}]\\
&=\bar v_{j}((\hat{\mu}-\epsilon) c-(e^{(\hat{\mu}-\epsilon)}-1+e^{-(\hat{\mu}-\epsilon)}))\\
&=\bar v_{j}(\hat{\mu}-\epsilon) \big(c-\frac{(e^{(\hat{\mu}-\epsilon)}-1+e^{-(\hat{\mu}-\epsilon)}))}{(\hat{\mu}-\epsilon)}\big )\\
&=\bar v_{j}(\hat{\mu}-\epsilon) \big(c-c(\hat{\mu}-\epsilon)\big )\\
&\geq 0.
\end{align*}

Choose $\epsilon =(\hat\mu-\mu)/3$ and we have $c>c(\hat{\mu}-\epsilon)=\frac{(e^{(\hat{\mu}-\epsilon)}-1+e^{-(\hat{\mu}-\epsilon)})}{(\hat{\mu}-\epsilon)}$. Thus, $\bar{v}$ is a super-solution of \eqref{partial-eq}.
In addition, for $j \geq M_0$, we have
$$\dot{u}_{j}-(u_{j+1}-2u_{j}+u_{j-1})=(f_j(u_{j})-f_j(0))u_j \leq 0.$$ 
Thus, $\underline{v}$ is a sub-solution of \eqref{partial-eq}.
By the comparison principle and letting $p \to \infty$, for $j>M_0+ct$, we have that
\begin{equation}
\label{uhat_estimate}
u_{j}(t) \leq K_{\epsilon} e^{-(\hat{\mu}-\epsilon)(j-ct)}.
\end{equation}
However, this contradicts with \eqref{u_bounded_estimate} by choosing $\epsilon =(\hat\mu-\mu)/3$.

\end{proof}

\begin{proof}[Proof of Theorem \ref{Existence and Nonexistence}]
\begin{itemize}
\item[(1)] Existence of transition fronts and tail estimates \eqref{tail_estimate_t0} have been shown in Propositions \ref{existence-prop1} - \ref{main-lm4}.
\item[(2)] Non-existence of transition fronts follows by Proposition \ref{Nonexistence-prop}.
\end{itemize}
\end{proof}

\section{Example}
In this section, we provide an example for localized perturbations in homogeneous media of \eqref{main-eq} with $f_j(u_j)=a_j-u_j$ and $a_j=1$ for $j \neq 0$. Thus \eqref{main-eq} becomes the following,
\vspace{-.1in}\begin{equation}
\label{main-example}
\ds\dot{u}_{j}=u_{j+1}-2u_{j}+u_{j-1}+a_ju_j(1-u_j),\quad j \in \ZZ,
\vspace{-.1in}\end{equation}
with $a_j=1$ for $j \neq 0$. The corresponding linearized equation is given by
\vspace{-.1in}\begin{equation}
\label{main-example-linear}
\ds\dot{u}_{j}=u_{j+1}-2u_{j}+u_{j-1}+a_ju_j,\quad j \in \ZZ.
\vspace{-.1in}\end{equation}
The eigenvalue problem is given by
\vspace{-.1in}\begin{equation}
\label{main-example-eigen}
\ds\lambda(\mu) u_j=e^{\mu}u_{j+1}-2u_{j}+e^{-\mu}u_{j-1}+a_ju_j,\quad j \in \ZZ.
\vspace{-.1in}\end{equation}

For homogeneous case, all $a_j$s are ones. By observation, $\lambda(\mu)=e^{\mu}-1+e^{-\mu}$ with constant eigenvector 1. It is easy to verify that $u_j(t)=e^{-\mu (j-ct)}$ is a solution of \eqref{main-example-linear} with $c=\frac{\lambda(\mu)}{\mu}$. Next we investigate the existence of the positive eigenvectors of \eqref{main-example-eigen} for the localized perturbation case $a_0 \neq 1$. We assume that one solution to localized perturbation case coincides with homogeneous case at the right with $u_j(t)=e^{-\mu (j-ct)}$ for $j \geq 0$. From \eqref{main-example-linear}, we have
\vspace{-.1in}\begin{equation}
\label{main-example-linear_iteration}
\ds u_{j-1}=\dot{u}_{j}+(2-a_j)u_{j}-u_{j+1},\quad j \in \ZZ.
\vspace{-.1in}\end{equation}

Thus, by induction, for $j<0$, 
\begin{align*}
 u_{-1}&=\dot{u}_{0}+(2-a_0)u_{0}-u_{1}=(1+(1-a_0)e^{-\mu})e^{-\mu(-1-ct)},\\
u_{-2}&=\dot{u}_{-1}+u_{-1}-u_{0}=(1+(1-a_0)e^{-\mu}+(1-a_0)e^{-3\mu})e^{-\mu(-2-ct)},\\
u_{-3}&=\dot{u}_{-2}+u_{-2}-u_{-1}=(1+(1-a_0)(e^{-\mu}+e^{-3\mu}+e^{-5\mu}))e^{-\mu(-3-ct)}, \\
&\vdots\\
u_{j}&=(1+(1-a_0)e^{-\mu}\frac{1-e^{2\mu j}}{1-e^{-2\mu}})e^{-\mu(j-ct)}.
\end{align*}

Therefore, the eigenvector to \eqref{main-example-eigen} is given by
 $\phi_{j}=(1+(1-a_0)e^{-\mu}\frac{1-e^{2\mu j}}{1-e^{-2\mu}})e^{-\mu j}$ for $j<0$ and $\phi_{j}=e^{-\mu j}$ for $j \geq 0$. Note that if $a_0 \leq 1$, $\phi_{j}>0$ for all $j \in \ZZ$. That means the positive eigenvectors always exist for $a_0 \leq 1$ and so do the transition fronts of speed c no less than $c^*$. The minimal speed $c^*$ is given by $\ds c^*=\frac{e^{\mu^*}+e^{-\mu^*}-1}{\mu^*}=\inf_{\mu>0}\frac{e^{\mu}+e^{-\mu}-1}{\mu} \approx 2.073$ at $\mu^* \approx 0.9071$.

For $a_0>1$, $\phi_{j}>0$ for all $j \in \ZZ$ whenever $a_0 \leq e^{\mu}-e^{-\mu}+1$, which implies that
 $\mu \geq ln[\frac{-(1-a_0)+\sqrt{(1-a_0)^2+4}}{2}]$ that gives the $\hat{\mu}=ln[\frac{-(1-a_0)+\sqrt{(1-a_0)^2+4}}{2}]$. By $\lambda=\lambda(\hat{\mu})$, we have $\lambda=\sqrt{(1-a_0)^2+4}-1$. On the interval [$\hat{\mu},\mu^{*}$] whenever $\hat{\mu} \leq \mu^{*}$, the speed is well defined by $c=\frac{\lambda(\mu)}{\mu}=\frac{e^{\mu}+e^{-\mu}-1}{\mu}$ and we can have both a minimal and a maximal speed on this closed interval. Outside this interval for $\mu < \hat{\mu}$, since the components of the eigenvector are mixed with negative and positive signs, we fail to obtain the transition fronts. If $\hat{\mu} > \mu^{*}$, that is, $a_0> e^{\mu^*}-e^{-\mu^*}+1$, then the existence interval [$\hat{\mu},\mu^{*}$] will be empty.

 In summary, we have the following facts,
\begin{itemize}[leftmargin=*]
\item[(1)] If $a_{0}\leq 1$, then $\lambda=1$. In this case, $\hat{c}=\infty$, that is, the existence interval of speeds for transition fronts is $[c^*,\infty)$.
\item[(2)] If $a_{0}>1$, then $\lambda>1$. If $\lambda < \lambda^{*}$, then transition front exists for any speed $c \in [c^{*}, \hat{c}]$. However, transition fronts do not exist under three cases: $c<c^*$, $c>\hat{c}$ and $\lambda>\lambda^*$. The $c^*$, $\hat{c}$, $\lambda$ and $\lambda^*$ are given by as follows:
\subitem[i] The minimal wave speed $c^*$ is given by $$c^*=\frac{e^{\mu^*}+e^{-\mu^*}-1}{\mu^*}=\inf_{\mu>0}\frac{e^{\mu}+e^{-\mu}-1}{\mu} \approx 2.073 \text{ at } \mu^* \approx 0.9071.$$
\subitem[ii] The maximal wave speed $\hat{c}$ is given by $$ \hat{c}=\frac{e^{\hat{\mu}}+e^{-\hat{\mu}}-1}{\hat{\mu}}\text{ at }\hat{\mu}=ln[\frac{-(1-a_0)+\sqrt{(1-a_0)^2+4}}{2}].$$ Note that $a_0>1$, $\hat{c}$ depending on $a_0$ is finite, that is, $\hat{c}< \infty$.
\subitem[iii] $\lambda=\sqrt{(1-a_0)^2+4}-1$ and $\lambda^* =e^{\mu^*}+e^{-\mu^*}-1 \approx 1.8808$. No transition fronts exist for $\lambda>\lambda^*$. Under this case, we must have $a_0>e^{\mu^*}-e^{-\mu^*}+1\approx 3.073$. 
\end{itemize}
\section{Concluding remarks}
We studied the existence and non-existence of transition fronts for monostable lattice differential equations in locally spatially inhomogeneous patchy environments. We collected fundamental tools such as discrete heat kernel estimates and discrete parabolic Harnack inequality. We proved that Poincar$\acute{e}$ inequality holds on a 2-regular graph and so does discrete parabolic Harnack inequality. Under the assumptions (H1)-(H2), there is a positive principal eigenvector for $\lambda^*>\lambda(\mu)>\lambda$. This positive principal eigenvector is the main ingredient in constructions of super/sub-solutions. The right end (i.e. $j>N$) of positive principal eigenvector is one, that coincides with the principal eigenvector in homogeneous media. There are significant differences on the middle localized perturbation part (i.e, $j \in [-N,N]$) and the left end (i.e. $j<-N$). However, this impact declines to $0$ for $j<-N$ as $|j| \to \infty$. With comparison principles and the super/sub-solutions, we obtained the transition fronts of mean wave speed on a finite range $(c^*,\hat c]$ and then pass the limit to have the case of $c=c^*$. For $c \in (c^*,\hat c]$, the profiles of transition fronts are highly related to the graphs of super-solutions $e^{-\mu(j-ct)}\phi_j^{\mu}$. Note that, in the right end for $j > N$, we have $e^{-\mu(j-ct)}$ that is moving at the exact speed $c$. For $j < N$, the profiles will change with amplitude $\phi_j^{\mu}$. If $c \in (c^*,\hat c)$ and $j << -N$, $\ds\lim_{j \to \infty}\phi_j^{\mu}=l>0$, that means they are essentially constant profiles $l e^{-\mu(j-ct)}$. For $c=\hat c$ and $j << -N$, $\ds\lim_{j \to \infty}\phi_j^{\mu}=0$, the amplitudes keep decreasing to $0$ as $j \to -\infty$.

We proved that the transit fronts if exist must own the exponential tail properties. There are no transition fronts at all if $\lambda>\lambda^*$, the mean wave speed is slower than the minimal speed $c^*$, or bigger than the maximal wave speed $\hat c$. The strongly localized spatial inhomogeneous patchy environments prevent the existence of transition fronts. The proof of minimal wave speed $c^*$ follows from the work of Shen and Kong (\cite{KoSh1,KoSh2}), where they also studied the localized perturbation with periodic media for both nonlocal problem and lattice differential equations. The proof of maximal wave speed $\hat c$ relies heavily on the fundamental tools discrete heat kernel estimates, comparison principles and discrete parabolic Harnack inequality. We leave the uniqueness and stability of transition fronts to \eqref{main-eq} \cite{VVZhang} and transition fronts to lattice differential equations with the localized perturbation of periodic media for future study. 

\bigskip

\noindent\textbf{Acknowledgements.} Van Vleck was supported in part by NSF grant DMS-1714195.

\end{document}